\documentclass[a4paper,11pt]{article}
\usepackage{amsmath,amsthm,amssymb,enumitem,xcolor}
\usepackage{bbm}
\usepackage[hang]{footmisc}

\usepackage[nosort,nocompress]{cite}

\usepackage[linktocpage=true,bookmarks=false,hyperfootnotes=false,colorlinks,
    linkcolor={red!60!black},
    citecolor={blue!50!black},
    urlcolor={blue!80!black}]{hyperref}

\renewcommand{\eqref}[1]{\hyperref[#1]{(\ref*{#1})}}

\newlist{enumlist}{enumerate}{1}
\setlist[enumlist]{labelindent=0cm,label={\rm (\arabic*)},labelwidth=3ex,labelsep=0.5ex,leftmargin=3.5ex,align=left,topsep=0ex,itemsep=0ex,parsep=0.5ex}

\newlist{itemlist}{itemize}{2}
\setlist[itemlist,1]{labelindent=0cm,label=$\bullet$,labelwidth=2.5ex,labelsep=0.5ex,leftmargin=3ex,align=left,topsep=0ex,itemsep=0ex,parsep=0.5ex}
\setlist[itemlist,2]{labelindent=0cm,label=--,labelwidth=2.5ex,labelsep=0.5ex,leftmargin=3ex,align=left,topsep=0ex,itemsep=0ex,parsep=0.3ex}

\numberwithin{equation}{section}

\newtheorem{thm}{Theorem}[section]
\newtheorem{thmstar}{Theorem}

\newtheorem{lemma}[thm]{Lemma}
\newtheorem{prop}[thm]{Proposition}

\theoremstyle{definition}
	\newtheorem{defi}[thm]{Definition}
	\newtheorem{rem}[thm]{Remark}

\newcommand{\C}{\mathbb{C}}

\newcommand{\Z}{\mathbb{Z}}
\newcommand{\N}{\mathbb{N}}

\newcommand{\T}{\mathbb{T}}
\newcommand{\F}{\mathbb{F}}

\newcommand{\cH}{\mathcal{H}}

\newcommand{\cF}{\mathcal{F}}
\newcommand{\cG}{\mathcal{G}}

\newcommand{\Ad}{\operatorname{Ad}}

\newcommand{\stab}{\operatorname{Stab}}
\newcommand{\Stab}{\operatorname{Stab}}

\newcommand{\Aut}{\operatorname{Aut}}
\newcommand{\id}{\mathord{\text{\rm id}}}
\newcommand{\tensor}{\mathbin{\overline{\otimes}}}

\newcommand{\Om}{\Omega}
\newcommand{\recht}{\rightarrow}
\newcommand{\SL}{\operatorname{SL}}
\newcommand{\be}{\beta}
\newcommand{\Ker}{\operatorname{Ker}}
\newcommand{\actson}{\curvearrowright}
\newcommand{\al}{\alpha}
\newcommand{\eps}{\varepsilon}
\newcommand{\GL}{\operatorname{GL}}
\newcommand{\cR}{\mathcal{R}}
\newcommand{\ot}{\otimes}
\newcommand{\cU}{\mathcal{U}}
\newcommand{\dpr}{^{\prime\prime}}
\newcommand{\ovt}{\mathbin{\overline{\otimes}}}
\newcommand{\om}{\omega}
\newcommand{\Crss}{\mathcal{C}_{\text{\rm rss}}}
\newcommand{\si}{\sigma}
\newcommand{\vphi}{\varphi}
\newcommand{\cZ}{\mathcal{Z}}
\newcommand{\etatil}{\widetilde{\eta}}
\newcommand{\cN}{\mathcal{N}}
\newcommand{\Soh}{\widehat{S_1}}
\newcommand{\Sth}{\widehat{S_2}}

\begin{document}

\begin{center}
{\boldmath\LARGE\bf A class of II$_1$ factors with exactly two \vspace{0.5ex}\\ group measure space decompositions}

\bigskip

{\sc by Anna Sofie Krogager and Stefaan Vaes{\renewcommand\thefootnote{}\footnote{\noindent KU~Leuven, Department of Mathematics, Leuven (Belgium).\\ E-mails: annasofie.krogager@wis.kuleuven.be and stefaan.vaes@wis.kuleuven.be.\\ Supported by European Research Council Consolidator Grant 614195, and by long term structural funding~-- Methusalem grant of the Flemish Government.}}}

%\bigskip
%
%To appear in {\it Journal de Math\'{e}matiques Pures et Appliqu\'{e}es}.
\end{center}

\begin{abstract}\noindent
We construct the first II$_1$ factors having exactly two group measure space decompositions up to unitary conjugacy. Also, for every positive integer $n$, we construct a II$_1$ factor $M$ that has exactly $n$ group measure space decompositions up to conjugacy by an automorphism.
\end{abstract}

\section{Introduction}

The \emph{group measure space construction} of Murray and von Neumann associates to every free ergodic probability measure preserving (pmp) group action $\Gamma \actson (X,\mu)$ a crossed product von Neumann algebra $L^\infty(X) \rtimes \Gamma$. The classification of these group measure space II$_1$ factors is one of the core problems in operator algebras. For $\Gamma$ infinite amenable, they are all isomorphic to the hyperfinite II$_1$ factor $R$, because by Connes' celebrated theorem \cite{Co75}, even all amenable II$_1$ factors are isomorphic with $R$.

For nonamenable groups $\Gamma$, rigidity phenomena appear and far reaching classification theorems for group measure space II$_1$ factors were established in Popa's deformation/rigidity theory, see e.g.\ \cite{Po01,Po03,Po04}. In these results, the subalgebra $A=L^\infty(X)$ of $M = A \rtimes \Gamma$ plays a special role. Indeed, if an isomorphism $\pi : A \rtimes \Gamma \recht B \rtimes \Lambda$ of group measure space II$_1$ factors satisfies $\pi(A) = B$, by \cite{Si55}, $\pi$ must come from an orbit equivalence between the group actions $\Gamma \actson A$ and $\Lambda \actson B$, so that methods from measured group theory can be used. The subalgebra $A \subset M$ is \emph{Cartan:} it is maximal abelian and the normalizer $\cN_M(A) = \{u \in \cU(M) \mid u A u^*\}$ generates $M$. One of the key results of \cite{Po01,Po03,Po04} is that for large classes of group actions, any isomorphism $\pi : A \rtimes \Gamma \recht B \rtimes \Lambda$ satisfies $\pi(A) = B$, up to unitary conjugacy.

In \cite{Pe09,PV09}, the most extreme form of rigidity, called \emph{W$^*$-superrigidity} was discovered: in certain cases, the crossed product II$_1$ factor $M = A \rtimes \Gamma$ entirely remembers $\Gamma$ and its action on $A$. These results were established by first proving that the II$_1$ factor $M$ has a \emph{unique group measure space Cartan subalgebra} up to unitary conjugacy and then proving that the group action $\Gamma \actson A$ is OE-superrigid. Since then, several uniqueness results for group measure space Cartan subalgebras were obtained, in particular \cite{Io10} proving that all Bernoulli actions $\Gamma \actson (X_0,\mu_0)^\Gamma$ of all infinite property~(T) groups are W$^*$-superrigid.

Note that a general Cartan subalgebra $A \subset M$ need not be of group measure space type, i.e.\ there need not exist a group $\Gamma$ complementing $A$ in such a way that $M = A \rtimes \Gamma$. This is closely related to the phenomenon that a countable pmp equivalence relation need not be the orbit equivalence relation of a group action that is free. The first actual uniqueness theorems for Cartan subalgebras up to unitary conjugacy were only obtained in \cite{OP07}, where it was proved in particular that $A$ is the unique Cartan subalgebra of $A \rtimes \Gamma$ whenever $\Gamma = \F_n$ is a free group and $\F_n \actson A$ is a free ergodic pmp action that is \emph{profinite}. More recently, in \cite{PV11}, it was shown that $A$ is the unique Cartan subalgebra of $A \rtimes \Gamma$ for \emph{arbitrary} free ergodic pmp actions of the free groups $\Gamma = \F_n$.

It is known since \cite{CJ82} that a II$_1$ factor $M$ may have two Cartan subalgebras $A,B \subset M$ that are non conjugate by an automorphism of $M$. Although several concrete examples of this phenomenon were given in \cite{OP08,PV09,SV11} and despite all the progress on uniqueness of Cartan subalgebras, there are so far no results describing \emph{all} Cartan subalgebras of a II$_1$ factor $M$ once uniqueness fails. In this paper, we prove such a result and the following is our main theorem.

\begin{thmstar}\phantomsection\label{thmA}
\begin{enumlist}
\item For every integer $n \geq 0$, there exist II$_1$ factors $M$ that have exactly $2^n$ group measure space Cartan subalgebras up to unitary conjugacy.
\item\label{thmA2} For every integer $n \geq 1$, there exist II$_1$ factors $M$ that have exactly $n$ group measure space Cartan subalgebras up to conjugacy by an automorphism of $M$.
\end{enumlist}
\end{thmstar}

Two free ergodic pmp actions are called \emph{W$^*$-equivalent} if they have isomorphic crossed product von Neumann algebras. Thus, a free ergodic pmp action $G \actson (X,\mu)$ is W$^*$-superrigid if every action that is W$^*$-equivalent to $G \actson (X,\mu)$ must be conjugate to $G \actson (X,\mu)$. Theorem \ref{thmA}\ref{thmA2} can then be rephrased in the following way: we construct free ergodic pmp actions $G \actson (X,\mu)$ with the property that $G \actson (X,\mu)$ is W$^*$-equivalent to exactly $n$ group actions, up to orbit equivalence of the actions (and actually also up to conjugacy of the actions, see Theorem \ref{thm.exactly-n}).

The II$_1$ factors $M$ in Theorem \ref{thmA} are concretely constructed as follows. Let $\Gamma$ be any torsion free nonelementary hyperbolic group and let $\beta : \Gamma \actson (A_0,\tau_0)$ be any trace preserving action on the amenable von Neumann algebra $(A_0,\tau_0)$ with $A_0 \neq \C 1$ and $\Ker \be \neq \{e\}$. We then define $(A,\tau) = (A_0,\tau_0)^\Gamma$ and consider the action $\sigma : \Gamma \times \Gamma \actson (A,\tau)$ given by $\sigma_{(g,h)}(\pi_k(a)) = \pi_{gkh^{-1}}(\beta_h(a))$ for all $g,h,k \in \Gamma$ and $a \in A_0$, where $\pi_k : A_0 \recht A$ denotes the embedding as the $k$'th tensor factor.

Our main result describes exactly all group measure space Cartan subalgebras of the crossed product $M = A_0^\Gamma \rtimes (\Gamma \times \Gamma)$.

\begin{thmstar}\label{thmB}
Let $M = A_0^\Gamma \rtimes (\Gamma \times \Gamma)$ be as above. Up to unitary conjugacy, all group measure space Cartan subalgebras $B \subset M$ are of the form $B = B_0^\Gamma$ where $B_0 \subset A_0$ is a group measure space Cartan subalgebra of $A_0$ with the following two properties: $\beta_g(B_0) = B_0$ for all $g \in \Gamma$ and $A_0$ can be decomposed as $A_0 = B_0 \rtimes \Lambda_0$ with $\beta_g(\Lambda_0) = \Lambda_0$ for all $g \in \Gamma$.
\end{thmstar}

In Section \ref{proof}, we actually prove a more general and more precise result, see Theorem \ref{main-SV}. In Section \ref{sec.examples}, we give concrete examples and computations, thus proving Theorem \ref{thmA} (see Theorem \ref{thm.exactly-n}).

Note that in Theorems \ref{thmA} and \ref{thmB}, we can only describe the group measure space Cartan subalgebras of $M$. The reason for this is that our method entirely relies on a technique of \cite{PV09}, using the so-called \emph{dual coaction} that is associated to a group measure space decomposition $M = B \rtimes \Lambda$, i.e.\ the normal $*$-homomorphism $\Delta : M \recht M \ovt M$ given by $\Delta(b v_s) = b v_s \ot v_s$ for all $b \in B$, $s \in \Lambda$. When $B \subset M$ is an arbitrary Cartan subalgebra, we do not have such a structural $*$-homomorphism.

Given a II$_1$ factor $M$ as in Theorem \ref{thmB} and given the dual coaction $\Delta : M \recht M \ovt M$ associated with an arbitrary group measure space decomposition $M = B \rtimes \Lambda$, Popa's key methods of malleability \cite{Po03} and spectral gap rigidity \cite{Po06b} for Bernoulli actions allow to prove that $\Delta(L(\Gamma \times \Gamma))$ can be unitarily conjugated into $M \ovt L(\Gamma \times \Gamma)$. An ultrapower technique of \cite{Io11}, in combination with the transfer-of-rigidity principle of \cite{PV09}, then shows that the ``mysterious'' group $\Lambda$ must contain two commuting nonamenable subgroups $\Lambda_1,\Lambda_2$. Note here that the same combination of \cite{Io11} and \cite{PV09} was recently used in \cite{CdSS15} to prove that if $\Gamma_1,\Gamma_2$ are nonelementary hyperbolic groups and $L(\Gamma_1 \times \Gamma_2) \cong L(\Lambda)$, then $\Lambda$ must be a direct product of two nonamenable groups. Once we know that $\Lambda$ contains two commuting nonamenable subgroups $\Lambda_1,\Lambda_2$, we use a combination of methods from \cite{Io10} and \cite{IPV} to prove that $\Lambda_1 \Lambda_2 \subset \Gamma \times \Gamma$. From that point on, it is not so hard any more to entirely unravel the structure of $B$ and $\Lambda$. Throughout these arguments, we repeatedly use the crucial dichotomy theorem of \cite{PV11,PV12} saying that hyperbolic groups $\Gamma$ are relatively strongly solid: in arbitrary tracial crossed products $M = P \rtimes \Gamma$, if a von Neumann subalgebra $Q \subset M$ is amenable relative to $P$, then either $Q$ embeds into $P$, or the normalizer of $Q$ stays amenable relative to $P$.

\section{Preliminaries}

\subsection{Popa's intertwining-by-bimodules}

We recall from \cite[Theorem 2.1 and Corollary 2.3]{Po03} Popa's method of intertwining-by-bimodules. When $(M,\tau)$ is a tracial von Neumann algebra and $P,Q\subset M$ are possibly non-unital von Neumann subalgebras, we write $P\prec_M Q$ if there exists a nonzero $P$-$Q$-subbimodule of $1_PL^2(M)1_Q$ that has finite right $Q$-dimension. We refer to \cite{Po03} for several equivalent formulations of this intertwining property. If $Pp\prec_M Q$ for all nonzero projections $p\in P'\cap 1_PM1_P$, we write $P\prec_M^f Q$.

We are particularly interested in the case where $M$ is a crossed product $M=A\rtimes\Gamma$ by a trace preserving action $\Gamma \actson (A,\tau)$. Given a subset $F\subset\Gamma$, we denote by $P_F$ the orthogonal projection of $L^2(M)$ onto the closed linear span of $\{au_g\mid a\in A, g\in F\}$, where $\{u_g\}_{g\in\Gamma}$ denote the canonical unitaries in $L(\Gamma)$. By
\cite[Lemma 2.5]{Va10}, a von Neumann subalgebra $P \subset pMp$ satisfies $P\prec_M^f A$ if and only if for every $\eps > 0$, there exists a finite subset $F\subset\Gamma$ such that
$$\|x-P_{F}(x)\|_2 \leq \|x\| \varepsilon\quad\text{for all }x\in P \; .$$

We also need the following elementary lemma.

\begin{lemma}\label{lem3}
Let $\Gamma\curvearrowright(A,\tau)$ be a trace preserving action and put $M=A\rtimes\Gamma$. If $P\subset M$ is a diffuse von Neumann subalgebra such that $P\prec^f A$, then $P\nprec L(\Gamma)$.
\end{lemma}

\begin{proof}
Let $\varepsilon>0$ be given. We have $P\prec^f A$. So, as explained above, we can take a finite set $F\subset\Gamma$ such that $\|u-P_F(u)\|_2\leq\frac{\varepsilon}{2}$ for all $u\in \mathcal{U}(P)$. Moreover, since $P$ is diffuse, we can choose a sequence of unitaries $(w_n)\subset \mathcal{U}(P)$ tending to 0 weakly. We will prove that $\|E_{L(\Gamma)}(xw_ny)\|_2\to0$ for all $x,y\in M$, meaning that $P\nprec L(\Gamma)$. Note that it suffices to consider $x,y\in (A)_1$.

Take $x,y\in(A)_1$. Write $w_n = \sum_{g \in \Gamma} (w_n)_g u_g$ with $(w_n)_g \in A$. Then,
$$\|E_{L(\Gamma)}(P_F(xw_ny))\|_2^2=\sum_{g\in F}|\tau(x(w_n)_g\sigma_g(y))|^2\to0 $$
since $w_n\to0$ weakly. Take $n_0$ large enough such that
$$\|E_{L(\Gamma)}(P_F(xw_ny))\|_2\leq\frac{\varepsilon}{2}\quad\text{for all }n\geq n_0 \; .$$
Since $P_F$ is $A$-$A$-bimodular, we have that
$$
\|E_{L(\Gamma)}(P_F(xw_ny))-E_{L(\Gamma)}(xw_ny)\|_2 \leq\|P_F(xw_ny)-xw_ny\|_2 =\|x(P_F(w_n)-w_n)y\|_2 \leq\frac{\varepsilon}{2}
$$
for all $n$. We conclude that $\|E_{L(\Gamma)}(xw_ny)\|_2\leq\varepsilon$ for all $n\geq n_0$.
\end{proof}

\subsection{Relative amenability}

A tracial von Neumann algebra $(M,\tau)$ is called \emph{amenable} if there exists a state $\vphi$ on $B(L^2(M))$ such that $\vphi|_M = \tau$ and such that $\vphi$ is $M$-central, meaning that $\varphi(xT) = \varphi(Tx)$ for all $x \in M$, $T \in B(L^2(M))$. In \cite[Section 2.2]{OP07}, the concept of relative amenability was introduced. The definition makes use of Jones' basic construction: given a tracial von Neumann algebra $(M,\tau)$ and a von Neumann subalgebra $P \subset M$, the basic construction $\langle M,e_P \rangle$ is defined as the commutant of the right $P$-action on $B(L^2(M))$. Following \cite[Definition 2.2]{OP07}, we say that a von Neumann subalgebra $Q \subset p M p$ is \emph{amenable relative to $P$ inside $M$} if there exists a positive functional $\vphi$ on $p \langle M,e_P \rangle p$ that is $Q$-central and satisfies $\vphi|_{pMp} = \tau$.

We say that $Q$ is \emph{strongly nonamenable relative to} $P$ if $Qq$ is nonamenable relative to $P$ for every nonzero projection $q \in Q'\cap pMp$. Note that in that case, also $qQq$ is strongly nonamenable relative to $P$ for all nonzero projections $q \in Q$.

We need the following elementary relationship between relative amenability and intertwining-by-bimodules.

\begin{prop}\label{prop-SV-rel-amen-intertw}
Let $(M,\tau)$ be a tracial von Neumann algebra and $Q, P_1, P_2 \subset M$ be von Neumann subalgebras with $P_1 \subset P_2$. Assume that $Q$ is strongly nonamenable relative to $P_1$. Then the following holds.
\begin{enumlist}
\item\label{one-rel-amen} If $Q \prec P_2$, there exist projections $q \in Q$, $p \in P_2$, a nonzero partial isometry $v \in q M p$ and a normal unital $*$-homomorphism $\theta : qQq \recht pP_2 p$ such that $x v = v \theta(x)$ for all $x \in qQq$ and such that, inside $P_2$, we have that $\theta(qQq)$ is nonamenable relative to $P_1$.
\item\label{two-rel-amen} We have $Q \not\prec P_1$.
\end{enumlist}
\end{prop}
\begin{proof}
\ref{one-rel-amen} Assume that $Q \prec P_2$. By \cite[Theorem 2.1]{Po03}, we can take projections $q \in Q$, $p \in P_2$, a nonzero partial isometry $v \in q M p$ and a normal unital $*$-homomorphism $\theta : qQq \recht pP_2 p$ such that $x v = v \theta(x)$ for all $x \in qQq$. Assume that $\theta(qQq)$ is amenable relative to $P_1$ inside $P_2$. We can then take a positive functional $\vphi$ on $p \langle P_2, e_{P_1} \rangle p$ that is $\theta(qQq)$-central and satisfies $\vphi|_{p P_2 p} = \tau$. Denote by $e_{P_2}$ the orthogonal projection of $L^2(M)$ onto $L^2(P_2)$. Observe that $e_{P_2} \langle M,e_{P_1} \rangle e_{P_2} = \langle P_2, e_{P_1} \rangle$. We can then define the positive functional $\om$ on $q \langle M,e_{P_1} \rangle q$ given by
$$\om(T) = \vphi(e_{P_2} v^* T v e_{P_2}) \quad\text{for all}\;\; T \in q \langle M,e_{P_1} \rangle q \; .$$
By construction, $\om$ is $qQq$-central and $\om(x) = \tau(v^* x v)$ for all $x \in q M q$. Writing $q_0 = vv^*$, we have $q_0 \in (Q' \cap M)q$ and it follows that $qQq q_0$ is amenable relative to $P_1$. This contradicts the strong nonamenability of $Q$ relative to $P_1$.

Finally, note that \ref{two-rel-amen} follows from \ref{one-rel-amen} by taking $P_1 = P_2$.
\end{proof}

\subsection{Properties of the dual coaction}

Let $M = B \rtimes \Lambda$ be any tracial crossed product von Neumann algebra. Denote by $\{v_s\}_{s \in \Lambda}$ the canonical unitaries and consider the \emph{dual coaction} $\Delta : M \recht M \ovt M$, i.e.\ the normal $*$-homomorphism defined by $\Delta(b) = b \ot 1$ for all $b \in B$ and $\Delta(v_s) = v_s \ot v_s$ for all $s \in \Lambda$.

First, we show the following elementary lemma.

\begin{lemma}\label{lem4}
A von Neumann subalgebra $A\subset B\rtimes\Lambda$ satisfies $\Delta(A)\subset A\tensor A$ if and only if $A=B_0\rtimes\Lambda_0$ for some von Neumann subalgebra $B_0\subset B$ and some subgroup $\Lambda_0<\Lambda$ that leaves $B_0$ globally invariant.
\end{lemma}

\begin{proof}
Let $A\subset B\rtimes\Lambda$ be a von Neumann subalgebra satisfying $\Delta(A)\subset A\tensor A$. Let $a\in A$ and write $a=\sum_{s\in\Lambda}a_sv_s$ with $a_s\in B$. Let $I=\{s\in\Lambda\mid a_s\neq0\}$. Fix $s\in I$ and define the normal linear functional $\omega$ on $B\rtimes\Lambda$ by $\omega(x)=\tau(xv_s^\ast a_s^\ast)$. Then $(\omega\otimes1)\Delta(a)=\|a_s\|_2^2\,v_s$. Since $\Delta(a)\in A\tensor A$, it follows that $v_s\in A$. Similarly, we define a linear functional $\rho$ on $B\rtimes\Lambda$ by $\rho(x)=\tau(xv_s^\ast)$. Then $(1\otimes\rho)\Delta(a)=a_sv_s\in A$ and it follows that $a_s\in A$. Since this holds for arbitrary $s\in I$, we conclude that $A=B_0\rtimes\Lambda_0$ where $B_0=A\cap B$ and $\Lambda_0=\{s\in\Lambda\mid v_s\in A\}$.
\end{proof}

The proof of the next result is almost identical to the proof of \cite[Lemma 7.2(4)]{IPV}. For the convenience of the reader, we provide some details.

\begin{prop}\label{prop4}
Assume that $(B,\tau)$ is amenable. If $Q\subset M$ is a von Neumann subalgebra without amenable direct summand, then $\Delta(Q)$ is strongly nonamenable relative to $M \ot 1$.
\end{prop}

\begin{proof}
Using the bimodule characterization of relative amenability (see \cite[Theorem 2.1]{OP07} and see also \cite[Proposition 2.4]{PV11}), it suffices to prove that the $(M \ovt M)$-$M$-bimodule
$$ _{M \ovt M} \bigl(L^2(M \ovt M) \underset{M \ot 1}{\ot} L^2(M \ovt M) \bigr)_{\Delta(M)}$$
is weakly contained in the coarse $(M \ovt M)$-$M$-bimodule. Denoting by $\sigma : M \ovt M \recht M \ovt M$ the flip automorphism, this bimodule is isomorphic with the $(M \ovt M)$-$M$-bimodule
$$ _{M \ovt M} L^2(M \ovt M \ovt M)_{(\id \ot \sigma)(\Delta(M) \ot 1)} \; .$$
So, it suffices to prove that the $M$-$M$-bimodule $_{M \ot 1} L^2(M \ovt M)_{\Delta(M)}$ is weakly contained in the coarse $M$-$M$-bimodule. Noting that this $M$-$M$-bimodule is isomorphic with a multiple of the $M$-$M$-bimodule $_M \bigl(L^2(M) \ot_B L^2(M)\bigr)_M$, the result follows from the amenability of $B$.
\end{proof}

\subsection{Spectral gap rigidity for co-induced actions}\label{coinduced}

Given a tracial von Neumann algebra $(A_0,\tau_0)$ and a countable set $I$, we denote by $(A_0,\tau_0)^I$ (or just $A_0^I$) the von Neumann algebra tensor product $\overline{\bigotimes}_I(A_0,\tau_0)$. For each $i\in I$, we denote by $\pi_i\colon A_0\to A_0^I$ the embedding of $A_0$ as the $i$'th tensor factor.

\begin{defi}
Let $\Gamma\curvearrowright I$ be an action of a countable group $\Gamma$ on a countable set $I$. We say that a trace preserving action $\sigma\colon\Gamma\curvearrowright(A_0,\tau_0)^I$ is \emph{built over} $\Gamma\curvearrowright I$ if it satisfies
$$\sigma_g(\pi_i(A_0)) = \pi_{g\cdot i}(A_0) \quad\text{for all}\;\; g\in\Gamma,\,i\in I \;. $$
\end{defi}

Assume that $\Gamma \actson A_0^I$ is an action built over $\Gamma \actson I$. Choose a subset $J \subset I$ that contains exactly one point in every orbit of $\Gamma \actson I$. For every $j \in J$, the group $\stab j$ globally preserves $\pi_j(A_0)$. This defines an action $\stab j \actson A_0$ that can be \emph{co-induced} to an action $\Gamma \actson A_0^{\Gamma / \stab i}$. The original action $\Gamma \actson A_0^I$ is conjugate with the direct product of all these co-induced actions. In particular, co-induced actions are exactly actions built over a transitive action $\Gamma \actson I = \Gamma / \Gamma_0$.

Popa's malleability \cite{Po03} and spectral gap rigidity \cite{Po06b} apply to actions built over $\Gamma \actson I$. The generalization provided in \cite[Theorems 3.1 and 3.3]{BV14} carries over verbatim and this gives the following result.

\begin{thm}\label{thm4-5}
Let $\Gamma\curvearrowright I$ be an action of an icc group on a countable set. Assume that $\stab\{i,j\}$ is amenable for all $i,j \in I$ with $i \neq j$.
Let $(A_0,\tau_0)$ be a tracial von Neumann algebra and $(N,\tau)$ a II$_1$ factor. Let $\Gamma\curvearrowright(A_0,\tau_0)^I$ be an action built over $\Gamma\curvearrowright I$ and put $M = A_0^I \rtimes \Gamma$.

If $P\subset N\tensor M$ is a von Neumann subalgebra that is strongly nonamenable relative to $N\tensor A_0^I$, then the relative commutant $Q := P' \cap N \ovt M$ satisfies at least one of the following properties:
\begin{enumlist}
\item there exists an $i \in I$ such that $Q \prec N \tensor (A_0^I\rtimes\stab i)$~;
\item there exists a unitary $v\in N\tensor M$ such that $v^\ast Qv\subset N\tensor L(\Gamma)$.
\end{enumlist}
\end{thm}

Recall that for a von Neumann subalgebra $P\subset M$, we define $\mathcal{QN}_M(P)\subset M$ as the set of elements $x\in M$ for which there exist $x_1,\ldots,x_n,y_1,\ldots,y_m\in M$ satisfying
$$xP\subset\sum_{i=1}^n Px_i\quad\text{and}\quad Px\subset\sum_{j=1}^m y_jP \; .$$
Then $\mathcal{QN}_M(P)$ is a $\ast$-subalgebra of $M$ containing $P$. Its weak closure is called the \emph{quasi-normalizer} of $P$ inside $M$.

The following lemma is proved in exactly the same way as \cite[Lemma 4.1]{IPV} and goes back to \cite[Theorem 3.1]{Po03}.

\begin{lemma}\label{lem7}
Let $\Gamma\curvearrowright I$ be an action. Let $A_0$ be a tracial von Neumann algebra and let $\Gamma\curvearrowright A_0^I$ be an action built over $\Gamma\curvearrowright I$. Consider the crossed product $M=A_0^I\rtimes\Gamma$ and let $(N,\tau)$ be an arbitrary tracial von Neumann algebra.
\begin{enumlist}
\item If $P\subset p(N\tensor L(\Gamma))p$ is a von Neumann subalgebra such that $P\nprec_{N\tensor L(\Gamma)} N\tensor L(\stab i)$ for all $i\in I$, then the quasi-normalizer of $P$ inside $p(N\tensor M)p$ is contained in $p(N\tensor L(\Gamma))p$.
\item Fix $i\in I$ and assume that $Q\subset q(N\tensor(A_0^I\rtimes\stab i))q$ is a von Neumann subalgebra such that for all $j\neq i$, we have $Q\nprec_{N\tensor(A_0^I\rtimes\stab i)}N\tensor(A_0^I\rtimes\stab \{i,j\})$. Then the quasi-normalizer of $Q$ inside $q(N\tensor M)q$ is contained in $q(N\tensor(A_0^I\rtimes\stab i))q$.
\end{enumlist}
\end{lemma}

\section{Transfer of rigidity}

Fix a trace preserving action $\Lambda\curvearrowright(B,\tau)$ and put $M=B\rtimes\Lambda$. We denote by $\{v_s\}_{s\in\Lambda}$ the canonical unitaries in $L(\Lambda)$.
Whenever $\mathcal{G}$ is a family of subgroups of $\Lambda$, we say that a subset $F\subset\Lambda$ is small relative to $\mathcal{G}$ if $F$ is contained in a finite union of subsets of the form $g\Sigma h$ where $g,h\in\Lambda$ and $\Sigma\in\mathcal{G}$.

Following the transfer of rigidity principle from \cite[Section 3]{PV09}, we prove the following theorem.

\begin{thm}\label{thm1a}
Let $\Lambda\curvearrowright(B,\tau)$ be a trace preserving action and put $M=B\rtimes\Lambda$. Let $\Delta\colon M\to M\tensor M$ be the dual coaction given by $\Delta(bv_s)=bv_s\otimes v_s$ for $b\in B, s\in\Lambda$. Let $\mathcal{G}$ be a family of subgroups of $\Lambda$.
Let $P,Q\subset M$ be two von Neumann subalgebras satisfying
\begin{enumlist}
\item $\Delta(P)\prec_{M\tensor M}M\tensor Q$,
\item $P\nprec_{M}B\rtimes\Sigma$ for all $\Sigma\in\mathcal{G}$.
\end{enumlist}
Then there exists a finite set $x_1,\ldots,x_n\in M$ and a $\delta>0$ such that the following holds: whenever $F\subset\Lambda$ is small relative to $\mathcal{G}$, there exists an element $s_F\in\Lambda-F$ such that
$$ \sum_{i,k=1}^n \|E_Q\left(x_iv_{s_F}x_k^\ast\right)\|_2^2\geq\delta \; . $$
\end{thm}

\begin{proof}
Since $\Delta(P)\prec_{M\tensor M} M\tensor Q$, we can find a finite set $x_1,\ldots,x_n\in M$ and $\rho>0$ such that
$$ \sum_{i,k=1}^n\|E_{M\tensor Q}\left((1\otimes x_i)\Delta(w)(1\otimes x_k^\ast)\right)\|_2^2\geq\rho\quad\text{for all }w\in \mathcal{U}(P) \; .$$
Given a subset $F\subset\Lambda$, we denote by $P_F$ the orthogonal projection of $L^2(M)$ onto the closed linear span of $\{bv_s\mid b\in B, s\in F\}$.
Since $P\nprec_M B\rtimes\Sigma$ for all $\Sigma\in\mathcal{G}$, it follows from \cite[Lemma 2.4]{Va10} that there exists a net of unitaries $\{w_j\}_{j\in J}\subset \mathcal{U}(P)$ such that $\|P_F(w_j)\|_2\to0$ for any set $F\subset\Lambda$ that is small relative to $\mathcal{G}$. For each $j\in J$, write $w_j=\sum_{s\in\Lambda}w^j_sv_s$ with $w_s^j\in B$ and compute
\begin{align*}
\rho \leq \sum_{i,k=1}^n\|E_{M\tensor Q}\left((1\otimes x_i)\Delta(w_j)(1\otimes x_k^\ast)\right)\|_2^2  &= \sum_{i,k=1}^n\left\|\sum_{s\in\Lambda}w^j_sv_s\otimes E_Q(x_iv_sx_k^\ast)\right\|_2^2 \\
&= \sum_{i,k=1}^n\sum_{s\in\Lambda}\|w^j_s\|_2^2\,\|E_Q(x_iv_sx_k^\ast)\|_2^2 \; .
\end{align*}
We now claim that the conclusion of the theorem holds with $\delta=\frac{\rho}{2}$. Indeed, assume for contradiction that there exists a subset $F\subset\Lambda$ that is small relative to $\mathcal{G}$ such that
$$\sum_{i,k=1}^n \|E_Q\left(x_iv_{s}x_k^\ast\right)\|_2^2<\delta\quad\text{for all }s\in\Lambda-F \; .$$
Put $K=\max\{\|x_i\|^2\|x_k^\ast\|_2^2 \mid i,k=1,\ldots,n\}$ and choose $j_0\in J$ such that $\|P_F(w_j)\|_2^2=\sum_{s\in F}\|w_s^j\|_2^2<\frac{\rho}{4Kn^2}$ for all $j\geq j_0$. We then get for $j \geq j_0$
\begin{align*}
\rho\leq \sum_{s \in \Lambda} \sum_{i,k=1}^n \|w^j_s\|_2^2\,\|E_Q(x_iv_sx_k^\ast)\|_2^2 &\leq n^2K\sum_{s\in F}\|w_s^j\|_2^2 + \sum_{s\in\Lambda-F}\|w_s^j\|_2^2\,\delta<\frac{\rho}{4}+\frac{\rho}{2} \; ,
\end{align*}
which is a contradiction.
\end{proof}

\section{Embeddings of group von Neumann algebras}\label{height}

Following \cite[Section 4]{Io10} and \cite[Section 3]{IPV}, we define the \emph{height} $h_\Gamma$ of an element in a group von Neumann algebra $L(\Gamma)$ as the absolute value of the largest Fourier coefficient, i.e.,
$$h_\Gamma(x)=\sup_{g\in\Gamma}|\tau(xu_g^\ast)|\quad\text{for}\;\; x\in L(\Gamma) \; . $$
Whenever $\cG \subset L(\Gamma)$, we write
$$h_\Gamma(\cG) = \inf \{h_\Gamma(x) \mid x \in \cG\} \; .$$
When $\Gamma$ is an icc group and $\Lambda$ is a countable group such that $L(\Lambda) = L(\Gamma)$ with $h_\Gamma(\Lambda) > 0$, it was proven in \cite[Theorem 3.1]{IPV} that there exists a unitary $u \in L(\Gamma)$ such that $u \T \Lambda u^* = \T \Gamma$. We need the following generalization. For this, recall that a unitary representation is said to be \emph{weakly mixing} if $\{0\}$ is the only finite dimensional subrepresentation.

\begin{thm}\label{thm.group-vnalg}
Let $\Gamma$ be a countable group and $p \in L(\Gamma)$ a projection. Assume that $\cG \subset \cU(p L(\Gamma) p)$ is a subgroup satisfying the following properties.
\begin{enumlist}
\item The unitary representation $\{\Ad v\}_{v \in \cG}$ on $L^2(p L(\Gamma) p \ominus \C p)$ is weakly mixing.
\item If $g \in \Gamma$ and $g \neq e$, then $\cG\dpr \not\prec L(C_\Gamma(g))$.
\item We have $h_\Gamma(\cG) > 0$.
\end{enumlist}
Then $p=1$ and there exists a unitary $u \in L(\Gamma)$ such that $u \cG u^* \subset \T \Gamma$.
\end{thm}
\begin{proof}
Write $M = L(\Gamma)$ and denote by $\Delta : M \recht M \ovt M : \Delta(u_g) = u_g \ot u_g$ the comultiplication on $L(\Gamma)$. We first prove that the unitary representation on $L^2(\Delta(p) (M \ovt M) \Delta(p) \ominus \Delta(\C p))$ given by $\{\Ad \Delta(v)\}_{v \in \cG}$  is weakly mixing. To prove this, assume that $\cH \subset L^2(\Delta(p) (M \ovt M) \Delta(p))$ is a finite dimensional subspace satisfying $\Delta(v) \cH \Delta(v^*) = \cH$ for all $v \in \cG$. Writing $P = \cG\dpr$, it follows that the closed linear span of $\cH \Delta(pMp)$ is a $\Delta(P)$-$\Delta(pMp)$-subbimodule of $L^2(\Delta(p) (M \ovt M) \Delta(p))$ that has finite right dimension. By \cite[Proposition 7.2]{IPV} (using that $P \not\prec L(C_\Gamma(g))$ for $g \neq e$), we get that $\cH \subset \Delta(L^2(pMp))$. Since the unitary representation $\{\Ad v\}_{v \in \cG}$ on $L^2(p M p \ominus \C p)$ is weakly mixing, we conclude that $\cH \subset \C \Delta(p)$.

Using the Fourier decomposition $v = \sum_{g \in \Gamma} (v)_g u_g$, we get for every $v \in \cG$ that
$$\tau((v \ot \Delta(v))(\Delta(v)^* \ot v^*)) = \sum_{g \in \Gamma} |(v)_g|^4 \geq h_\Gamma(v)^4 \geq h_\Gamma(\cG)^4 \; .$$
Defining $X \in M \ovt M \ovt M$ as the element of minimal $\|\,\cdot\,\|_2$ in the weakly closed convex hull of $\{(v \ot \Delta(v))(\Delta(v)^* \ot v^*) \mid v \in \cG\}$, we get that $\tau(X) \geq h_\Gamma(\cG)^4$, so that $X$ is nonzero, and that $(v \ot \Delta(v)) X = X (\Delta(v) \ot v)$ for all $v \in \cG$. Also note that $(p \ot \Delta(p))X = X = X(\Delta(p) \ot p)$. By the weak mixing of both $\Ad v$ and $\Ad \Delta(v)$, it follows that $XX^*$ is multiple of $p \ot \Delta(p)$ and that $X^* X$ is a multiple of $\Delta(p) \ot p$. We may thus assume that
$$XX^* = p \ot \Delta(p) \quad\text{and}\quad X^* X = \Delta(p) \ot p \; .$$
Define $Y = (1 \ot X)(X \ot 1)$. Note that $Y \in M \ovt M \ovt M \ovt M$ is a partial isometry with $YY^* = p \ot p \ot \Delta(p)$ and $Y^* Y = \Delta(p) \ot p \ot p$. Also,
$$Y = (v \ot v \ot \Delta(v)) Y (\Delta(v)^* \ot v^* \ot v^*) \quad\text{for all}\;\; v \in \cG \; .$$
Since $Y$ is nonzero, it follows that the unitary representation $\xi \mapsto (v \ot v)\xi \Delta(v^*)$ of $\cG$ on the Hilbert space $(p \ot p)L^2(M \ovt M) \Delta(p)$ is not weakly mixing. We thus find a finite dimensional irreducible representation $\om : \cG \recht \cU(\C^n)$ and a nonzero $Z \in M_{n,1}(\C) \ot (p \ot p)L^2(M \ovt M) \Delta(p)$ satisfying
$$(\om(v) \ot v \ot v) Z = Z \Delta(v)  \quad\text{for all}\;\; v \in \cG \; .$$
By the weak mixing of $\Ad v$ and $\Ad \Delta(v)$ and the irreducibility of $\om$, it follows that $Z Z^*$ is a multiple of $1 \ot p \ot p$ and that $Z^* Z$ is a multiple of $\Delta(p)$. So, we may assume that $ZZ^* = 1 \ot p \ot p$ and that $Z^* Z = \Delta(p)$. It follows that $Z^* (M_n(\C) \ot p \ot p) Z$ is an $n^2$-dimensional globally $\{\Ad \Delta(v)\}_{v \in \cG}$-invariant subspace of $\Delta(p)(M \ovt M)\Delta(p)$. Again by weak mixing, this implies that $n = 1$. But then, since $\tau(ZZ^*) = \tau(Z^* Z)$, we also get that $p=1$. So, $Z \in M \ovt M$ is a unitary operator and $\om : \cG \recht \T$ is a character satisfying $\om(v) (v \ot v)Z = Z \Delta(v)$ for all $v \in \cG$.

Denoting by $\sigma : M \ovt M \recht M \ovt M$ the flip map and using that $\sigma \circ \Delta = \Delta$, it follows that $Z \sigma(Z)^*$ commutes with all $v \ot v$, $v \in \cG$. By weak mixing, $Z \sigma(Z)^*$ is a multiple of $1$. Using that $(\Delta \ot \id) \circ \Delta = (\id \ot \Delta) \circ \Delta$, we similarly find that $(Z \ot 1) (\Delta \ot \id)(Z)$ is a multiple of $(1 \ot Z)(\id \ot \Delta)(Z)$. By \cite[Theorem 3.3]{IPV}, there exists a unitary $u \in M$ such that $Z = (u^* \ot u^*)\Delta(u)$. But then,
$$\Delta(u v u^*) = \om(v) \, u v u^* \ot u v u^* \quad\text{for all}\;\; v \in \cG \; .$$
This means that $u v u^* \in \T \Gamma$ for every $v \in \cG$.
\end{proof}

\section{Proof of Theorem \ref{thmB}}\label{proof}

As in \cite[Definition 2.7]{CIK13}, we consider the class $\Crss$ of \emph{relatively strongly solid} groups consisting of all nonamenable countable groups $\Gamma$ such that for any tracial crossed product $M=P\rtimes\Gamma$ and any von Neumann subalgebra $Q\subset pMp$ that is amenable relative to $P$, we have that either $Q\prec P$ or the normalizer $\mathcal{N}_{pMp}(Q)''$ stays amenable relative to $P$.

The class $\Crss$ is quite large. Indeed, by \cite[Theorem 1.6]{PV11}, all weakly amenable groups that admit a proper $1$-cocycle into an orthogonal representation weakly contained in the regular representation belong to $\Crss$. In particular, the free groups $\F_n$ with $2\leq n\leq\infty$ belong to $\Crss$ and more generally, all free products $\Lambda_1\ast\Lambda_2$ of amenable groups $\Lambda_1,\Lambda_2$ with $|\Lambda_1|\geq2$ and $|\Lambda_2|\geq3$ belong to $\Crss$. By \cite[Theorem 1.4]{PV12}, all weakly amenable, nonamenable, bi-exact groups belong to $\Crss$ and thus $\Crss$ contains all nonelementary hyperbolic groups.

Theorem \ref{thmB} is an immediate consequence of the more general Theorem \ref{main-SV} that we prove in this section. In order to make our statements entirely explicit, we call a group measure space (gms) decomposition of a tracial von Neumann algebra $(M,\tau)$ any pair $(B,\Lambda)$ where $B \subset M$ is a maximal abelian von Neumann subalgebra and $\Lambda \subset \cU(M)$ is a subgroup normalizing $B$ such that $M = (B \cup \Lambda)\dpr$ and $E_B(v) = 0$ for all $v \in \Lambda \setminus \{1\}$. This of course amounts to writing $M = B \rtimes \Lambda$ for some free and trace preserving action $\Lambda \actson (B,\tau)$.

We then say that two gms decompositions $(B_i,\Lambda_i)$, $i=0,1$, of $M$ are
\begin{itemlist}
\item identical if $B_0=B_1$ and $\T\Lambda_0=\T\Lambda_1$~;
\item unitarily conjugate if there exists a unitary $u\in\mathcal{U}(M)$ such that $uB_0u^\ast=B_1$ and $u\T\Lambda_0 u^\ast=\T\Lambda_1$;
\item conjugate by an automorphism if there exists an automorphism $\theta\in\Aut(M)$ such that $\theta(B_0)=B_1$ and $\theta(\T\Lambda_0)=\T\Lambda_1$.
\end{itemlist}

\begin{thm}\label{main-SV}
Let $\Gamma$ be a torsion free group in the class $\Crss$. Let $(A_0,\tau_0)$ be any amenable tracial von Neumann algebra with $A_0 \neq \C 1$ and $\be : \Gamma \actson (A_0,\tau_0)$ any trace preserving action such that $\Ker \be$ is a nontrivial subgroup of $\Gamma$. Define $(A,\tau) = (A_0,\tau_0)^\Gamma$ and denote by $\pi_k : A_0 \recht A$ the embedding as the $k$'th tensor factor. Define the action $\si: \Gamma \times \Gamma \actson (A,\tau)$ given by $\si_{(g,h)}(\pi_k(a)) = \pi_{gkh^{-1}}(\be_h(a))$ for all $g,k,h \in \Gamma$ and $a \in A_0$. Denote $M = A \rtimes (\Gamma \times \Gamma)$.

Up to unitary conjugacy, all gms decompositions of $M$ are given as $M = B \rtimes \Lambda$ with $B = B_0^\Gamma$ and $\Lambda = \Lambda_0^{(\Gamma)} \rtimes (\Gamma \times \Gamma)$ where $A_0 = B_0 \rtimes \Lambda_0$ is a gms decomposition of $A_0$ satisfying $\beta_g(B_0) = B_0$ and $\beta_g(\Lambda_0) = \Lambda_0$ for all $g \in \Gamma$.

Moreover, the gms decompositions of $M$ associated with $(B_0,\Lambda_0)$ and $(B_1,\Lambda_1)$ are
\begin{enumlist}
\item\label{main.uc} unitarily conjugate iff $(B_0,\Lambda_0)$ is identical to $(B_1,\Lambda_1)$,
\item\label{main.autc} conjugate by an automorphism of $M$ iff there exists a trace preserving automorphism $\theta_0 : A_0 \recht A_0$ and an automorphism $\vphi \in \Aut(\Gamma)$ such that $\theta_0(B_0) = B_1$, $\theta_0(\T \Lambda_0) = \T \Lambda_1$ and $\theta_0 \circ \be_g = \be_{\vphi(g)} \circ \theta_0$ for all $g \in \Gamma$.
\end{enumlist}
\end{thm}

Note that in Proposition \ref{prop-final} at the end of this section, we discuss when the Cartan subalgebras $B = B_0^\Gamma$ are unitarily conjugate, resp.\ conjugate by an automorphism of $M$.

\begin{lemma}\label{lem1}
Let $\Gamma$ be a group in $\Crss$ and $M=P\rtimes\Gamma$ any tracial crossed product. If $Q_1,Q_2 \subset pMp$ are commuting von Neumann subalgebras, then either $Q_1 \prec_M P$ or $Q_2$ is amenable relative to $P$.
\end{lemma}
\begin{proof}
Assume that $Q_1 \not\prec_M P$. By \cite[Corollary F.14]{BO08}, there exists a diffuse abelian von Neumann subalgebra $A\subset Q_1$ such that $A\nprec_M P$. Because $\Gamma \in \Crss$, we get that $\mathcal{N}_{pMp}(A)''$ is amenable relative to $P$. Since $Q_2 \subset \mathcal{N}_{pMp}(A)''$, also $Q_2$ is amenable relative to $P$.
\end{proof}

It also follows that for groups $\Gamma$ in $\Crss$, the centralizer $C_\Gamma(L)$ of an infinite subgroup $L < \Gamma$ is amenable. So, torsion free groups $\Gamma$ in $\Crss$ have the property that $C_\Gamma(g)$ is amenable for every $g \neq e$. As a consequence, torsion free groups $\Gamma$ in $\Crss$ are icc and even have the property that every nonamenable subgroup $L < \Gamma$ is \emph{relatively icc} in the sense that $\{h g h^{-1} \mid h \in L\}$ is an infinite set for every $g \in \Gamma$, $g \neq e$. Finally note that torsion free groups $\Gamma$ in $\Crss$ have no nontrivial amenable normal subgroups. In particular, every nontrivial normal subgroup of $\Gamma$ is relatively icc.

In the rest of this section, we prove Theorem \ref{main-SV}. So, we fix a group $\Gamma$ and an action $\Gamma \times \Gamma \actson A$ as in the formulation of the theorem. We put $M = A \rtimes (\Gamma \times \Gamma)$.

\begin{lemma}\label{lem1-SV}
Let $(N,\tau)$ be a tracial factor and let $Q_1,Q_2 \subset N \ovt M$ be commuting von Neumann subalgebras that are strongly nonamenable relative to $N \ot 1$. Then $Q_1 \vee Q_2$ can be unitarily conjugated into $N \ovt L(\Gamma \times \Gamma)$.
\end{lemma}
\begin{proof}
Since $A$ is amenable, we get that $Q_1$ and $Q_2$ are strongly nonamenable relative to $N \ovt (A \rtimes L)$ whenever $L < \Gamma \times \Gamma$ is an amenable subgroup. For every $g \in \Gamma$, we denote by $\Stab g \subset \Gamma \times \Gamma$ the stabilizer of $g$ under the left-right translation action $\Gamma \times \Gamma \actson \Gamma$. We also write $\Stab \{g,h\} = \Stab g \cap \Stab h$.

We start by proving that $Q_2 \not\prec N \ovt (A \rtimes \Stab g)$ for all $g \in \Gamma$. Assume the contrary. Whenever $h \neq g$, the group $\Stab \{g,h\}$ is amenable so that by Proposition \ref{prop-SV-rel-amen-intertw}, $Q_2 \not\prec N \ovt (A \rtimes \Stab \{g,h\})$. Also by Proposition \ref{prop-SV-rel-amen-intertw}, we can take projections $q \in Q_2$ and $p \in N \ovt (A \rtimes \Stab g)$, a nonzero partial isometry $v \in q (N \ovt M) p$ and a normal unital $*$-homomorphism
$$\theta : q Q_2 q \recht p(N \ovt (A \rtimes \Stab g))p$$
such that $x v = v \theta(x)$ for all $x \in qQ_2 q$ and such that, inside $N \ovt (A \rtimes \Stab g)$, we have that $\theta(qQ_2q)$ is nonamenable relative to $N \ovt A$ and we have that $\theta(qQ_2q) \not\prec N \ovt (A \rtimes \Stab \{g,h\})$ whenever $h \neq g$.

Write $P := \theta(q Q_2 q)' \cap p (N \ovt M) p$. By Lemma \ref{lem7}, $P \subset p (N \ovt (A \rtimes \Stab g)) p$. In particular, $v^* v \in p (N \ovt (A \rtimes \Stab g)) p$ and we may assume that $v^* v = p$. Since $\Stab g \cong \Gamma$, we have $\Stab g \in \Crss$ and Lemma \ref{lem1} implies that $P \prec N \ovt A$. Conjugating with $v$ and writing $e = vv^* \in (Q_2' \cap (N \ovt M))q$, we find that $e(Q_2' \cap (N \ovt M))e \prec N \ovt A$. Since $Q_1 \subset Q_2' \cap (N \ovt M)$, it follows that $Q_1 \prec N \ovt A$. By Proposition \ref{prop-SV-rel-amen-intertw}, this contradicts the strong nonamenability of $Q_1$ relative to $N \ovt A$. So, we have proved that $Q_2 \not\prec N \ovt (A \rtimes \Stab g)$ for all $g \in \Gamma$.

Since $Q_1$ is strongly nonamenable relative to $N \ovt A$ and since $Q_2 \not\prec N \ovt (A \rtimes \Stab g)$ for all $g \in \Gamma$, it follows from Theorem \ref{thm4-5} that $v^* Q_2 v \subset N \ovt L(\Gamma \times \Gamma)$ for some unitary $v \in N \ovt M$. Since $v^* Q_2 v \not\prec N \ovt L(\Stab g)$ for all $g \in \Gamma$, it follows from Lemma \ref{lem7} that also $v^* Q_1 v \subset N \ovt L(\Gamma \times \Gamma)$. This concludes the proof of the lemma.
\end{proof}

We now also fix a gms decomposition $M = B \rtimes \Lambda$. We view $\Lambda$ as a subgroup of $\cU(M)$. We denote by $\Delta : M \recht M \ovt M$ the associated dual coaction given by $\Delta(b) = b \ot 1$ for all $b \in B$ and $\Delta(v) = v \ot v$ for all $v \in \Lambda$.

\begin{lemma}\label{lem2-SV}
Writing $Q_1 = L(\Gamma \times \{e\})$ and $Q_2 = L(\{e\} \times \Gamma)$, we have $\Delta(Q_2)\prec_{M\tensor M} M\tensor Q_i$ for either $i=1$ or $i=2$.
\end{lemma}

\begin{proof}
By Proposition \ref{prop4}, $\Delta(Q_1)$ and $\Delta(Q_2)$ are strongly nonamenable relative to $M \ot 1$. So by Lemma \ref{lem1-SV}, we can take a unitary $v \in M \ovt M$ such that
$$v^\ast\Delta(Q_1 \vee Q_2)v\subset M\tensor L(\Gamma\times\Gamma) \; .$$
We therefore have the two commuting subalgebras $v^\ast\Delta(Q_1)v$ and $v^\ast\Delta(Q_2)v$ inside $M\tensor L(\Gamma\times\Gamma)$. If $v^\ast\Delta(Q_1)v$ was amenable relative to both $M\tensor Q_1$ and $M\tensor Q_2$, then it would be amenable relative to $M\tensor 1$ by \cite[Proposition 2.7]{PV11}, which is not the case. Hence $v^\ast\Delta(Q_1)v$ is nonamenable relative to either $M\tensor Q_1$ or $M\tensor Q_2$. Assuming that $v^\ast\Delta(Q_1)v$ is nonamenable relative to $M\tensor Q_1$, Lemma \ref{lem1} implies that $\Delta(Q_2)\prec M\tensor Q_1$.
\end{proof}

In the following three lemmas, we prove that $\Lambda$ contains two commuting nonamenable subgroups $\Lambda_1,\Lambda_2 < \Lambda$. The method to produce such commuting subgroups is taken from \cite{Io11} and our proofs of Lemmas \ref{lem3-SV}, \ref{lem4-SV} and \ref{lem5-SV} are very similar to the proof of \cite[Theorem 3.1]{Io11}. The same method was also used in \cite[Theorem 3.3]{CdSS15}. For completeness, we provide all details.

Combining Theorem \ref{thm1a} and Lemma \ref{lem2-SV}, we get the following.

\begin{lemma}\label{lem3-SV}
Denote by $\mathcal{G}$ the family of all amenable subgroups of $\Lambda$. For either $i=1$ or $i=2$, there exists a finite set $x_1,\ldots,x_n\in M$ and a $\delta>0$ such that the following holds: whenever $F\subset\Lambda$ is small relative to $\mathcal{G}$, we can find an element $v_F \in\Lambda-F$ such that
$$\sum_{k,j=1}^n \|E_{Q_i}(x_k v_F x_j^\ast)\|_2^2\geq\delta \; . $$
\end{lemma}

\begin{lemma}\label{lem4-SV}
There exists a decreasing sequence of nonamenable subgroups $\Lambda_n<\Lambda$ such that $Q_i\prec_M B\rtimes(\bigcup_n C_\Lambda(\Lambda_n))$ for either $i=1$ or $i=2$, where $C_\Lambda(\Lambda_n)$ denotes the centralizer of $\Lambda_n$ inside $\Lambda$.
\end{lemma}

\begin{proof}
As in Lemma \ref{lem3-SV}, we let $\mathcal{G}$ denote the family of all amenable subgroups of $\Lambda$. We denote by $I$ the set of subsets of $\Lambda$ that are small relative to $\mathcal{G}$. We order $I$ by inclusion and choose a cofinal ultrafilter $\mathcal{U}$ on $I$. Consider the ultrapower von Neumann algebra $M^{\mathcal{U}}$ and the ultrapower group $\Lambda^{\mathcal{U}}$. Every $v=(v_F)_{F\in I}\in\Lambda^{\mathcal{U}}$ can be viewed as a unitary in $M^{\mathcal{U}}$.

Assume without loss of generality that $i=1$ in Lemma \ref{lem3-SV} and denote by $v=(v_F)_{F\in I}$ the element of $\Lambda^{\mathcal{U}}$ that we found in Lemma \ref{lem3-SV}. Denote by $K\subset L^2(M^{\mathcal{U}})$ the closed linear span of $MvM$ and by $P_K$ the orthogonal projection from $L^2(M^{\mathcal{U}})$ onto $K$. Put $\Sigma=\Lambda\cap v\Lambda v^{-1}$. We claim that $Q_2\prec_M B\rtimes\Sigma$.

Assume the contrary. This means that we can find a sequence of unitaries $a_n\in \mathcal{U}(Q_2)$ such that $\|E_{B\rtimes\Sigma}(xa_ny)\|_2\to0$ for any $x,y\in M$. We prove that $\langle a_n\xi a_n^\ast,\eta\rangle\to0$ as $n\to\infty$ for all $\xi,\eta\in K$. For this, it suffices to prove that $\langle a_n xv x'a_n^\ast,yv y'\rangle\to0$ for all $x,x',y,y'\in M$. First, note that for all $z\in M$, we have $E_M(v^\ast zv )=E_M(v^\ast E_{B\rtimes\Sigma}(z)v)$ by definition of the subgroup $\Sigma$. Hence
\begin{align*}
\langle a_n xvx'a_n^\ast,yv y'\rangle &= \tau(E_M(v^\ast y^\ast a_nxv)x'a_n^\ast y'^\ast) \\
&= \tau(E_M(v^\ast E_{B\rtimes\Sigma}(y^\ast a_nx)v)x'a_n^\ast y'^\ast)\\
&\leq\|x'\|\|y'\|\|E_{B\rtimes\Sigma}(y^\ast a_nx)\|_2\to0
\end{align*}
as wanted.

Next, Lemma \ref{lem3-SV} provides a finite set $L\subset M$ such that $\sum_{x,y\in L}\|E_{Q_1^{\mathcal{U}}}(xvy^\ast)\|^2_2\neq0$. In particular, we can take $x,y\in L$ such that $E_{Q_1^{\mathcal{U}}}(xvy^\ast)\neq0$. Put $\xi=P_K(E_{Q_1^{\mathcal{U}}}(xvy^\ast))$. We claim that $\xi\neq0$. Since $E_{Q_1^{\mathcal{U}}}(xvy^\ast)\neq0$, we get that $\|xv y^\ast-E_{Q_1^{\mathcal{U}}}(xv y^\ast)\|_2<\|xv y^\ast\|_2$. Since $xv y^\ast\in K$, it follows that $\|xvy^\ast-\xi\|_2=\|P_K(xvy^\ast-E_{Q_1^{\mathcal{U}}}(xvy^\ast))\|_2<\|xvy^\ast\|_2$. Hence $\xi\neq0$.

Since $K$ is an $M$-$M$-bimodule and since $Q_1$ commutes with $Q_2$, we have that $a\xi=\xi a$ for all $a\in Q_2$. In particular, $\langle a_n\xi a_n^\ast,\xi\rangle=\|\xi\|_2^2>0$ in contradiction with the fact that $\langle a_n\xi a_n^\ast,\xi\rangle\to0$. This proves that $Q_2\prec_M B\rtimes\Sigma$.

It remains to show that there exists a decreasing sequence of subgroups $\Lambda_n<\Lambda$ such that for all $n$ we have $\Lambda_n\notin\mathcal{G}$, and such that $\Sigma=\bigcup_n C_\Lambda(\Lambda_n)$. For every $\mathcal{T}\subset I$, we denote by $\Lambda_\mathcal{T}$ the subgroup of $\Lambda$ generated by $\{v_Fv_{F'}^{-1}\mid F,F'\in \mathcal{T}\}$. An element $w\in\Lambda$ belongs to $\Sigma$ if and only if there exists a $\mathcal{T}\in\mathcal{U}$ such that $w$ commutes with $\Lambda_\mathcal{T}$. Enumerating $\Sigma=\{w_1,w_2,\ldots\}$, choose $\mathcal{S}_n\in\mathcal{U}$ such that $w_n$ commutes with $\Lambda_{\mathcal{S}_n}$. Then put $\mathcal{T}_n:=\mathcal{S}_1\cap\ldots\cap\mathcal{S}_n$. Note that $\mathcal{T}_n\in\mathcal{U}$ and by construction, $\Sigma=\bigcup_n C_\Lambda(\Lambda_{\mathcal{T}_n})$. It remains to prove that $\Lambda_\mathcal{T}\notin\mathcal{G}$ for all $\mathcal{T}\in\mathcal{U}$.

Fix $\mathcal{T} \in \mathcal{U}$ and assume that $\Lambda_\mathcal{T} \in \mathcal{G}$. Fix an element $F'\in\mathcal{T}$. Then $\{v_F\mid F\in\mathcal{T}\}\subset\Lambda_\mathcal{T}v_{F'}$. So, $F_1:=\{v_F\mid F\in\mathcal{T}\}$ is small relative to $\mathcal{G}$. Define $\mathcal{T}'\subset I$ by $\mathcal{T}'=\{F\in I\mid F_1\subset F\}$. Since $\mathcal{U}$ is a cofinal ultrafilter and $\mathcal{T} \in \mathcal{U}$, we get $\mathcal{T}\cap\mathcal{T}'\neq\emptyset$. So we can take $F\in\mathcal{T}$ with $F_1\subset F$. Then, $v_F\in\Lambda-F\subset\Lambda-F_1$ but also $v_F\in F_1$. This being absurd, we have shown that $\Lambda_\mathcal{T}\notin\mathcal{G}$ for all $\mathcal{T} \in \mathcal{U}$.
\end{proof}

\begin{lemma}\label{lem5-SV}
There exist two commuting nonamenable subgroups $\Lambda_1$ and $\Lambda_2$ inside $\Lambda$. Moreover, whenever $\Lambda_1,\Lambda_2 < \Lambda$ are commuting nonamenable subgroups, $L(\Lambda_1\Lambda_2)$ can be unitarily conjugated into $L(\Gamma\times\Gamma)$.
\end{lemma}

\begin{proof}
From Lemma \ref{lem4-SV}, we find a decreasing sequence of nonamenable subgroups $\Lambda_n<\Lambda$ such that $Q_i\prec_M B\rtimes(\bigcup_n C_\Lambda(\Lambda_n))$ for either $i=1$ or $i=2$. Since $Q_i$ has no amenable direct summand, we get that the group $\bigcup_n C_\Lambda(\Lambda_n)$ is nonamenable. It follows that $C_\Lambda(\Lambda_n)$ is nonamenable for some $n\in\N$. Denote $\Lambda_1:=\Lambda_n$ and $\Lambda_2:=C_\Lambda(\Lambda_n)$.

When $\Lambda_1,\Lambda_2 < \Lambda$ are commuting nonamenable subgroups, it follows from Lemma \ref{lem1-SV} applied to $N=\C 1$ that $L(\Lambda_1) \vee L(\Lambda_2)$ can be unitarily conjugated into $L(\Gamma \times \Gamma)$.
\end{proof}

From now on, we fix commuting nonamenable subgroups $\Lambda_1,\Lambda_2 < \Lambda$. By Lemma \ref{lem5-SV}, after a unitary conjugacy, we may assume that $L(\Lambda_1 \Lambda_2) \subset L(\Gamma \times \Gamma)$.

\begin{lemma}\label{lem6}
If $N\subset L(\Gamma\times\Gamma)$ is an amenable von Neumann subalgebra such that the normalizer $\mathcal{N}_{L(\Gamma\times\Gamma)}(N)''$ contains $L(\Lambda_1 \Lambda_2)$, then $N$ is atomic. Also, $L(\Lambda_1 \Lambda_2)' \cap L(\Gamma \times \Gamma)$ is atomic.
\end{lemma}

\begin{proof}
Using \cite[Proposition 2.6]{Va10}, we find a projection $q$ in the center of the normalizer of $N$ such that $Nq\prec^f L(\Gamma)\otimes1$ and $N(1-q)\nprec L(\Gamma)\otimes1$.

Assume for contradiction that $q\neq1$. Since $N(1-q)\nprec L(\Gamma)\otimes1$ and since $\Gamma \in \Crss$, it follows that $L(\Lambda_i)(1-q)$ is amenable relative to $L(\Gamma)\otimes1$ for both $i=1,2$. It then follows from \cite[Proposition 2.7]{PV11} that $L(\Lambda_1)(1-q)$ is nonamenable relative to $1\otimes L(\Gamma)$, hence $L(\Lambda_2)(1-q)\prec1\otimes L(\Gamma)$ by Lemma \ref{lem1}. By Proposition \ref{prop-SV-rel-amen-intertw}, we get a nonzero projection $q_0 \leq 1-q$ that commutes with $L(\Lambda_2)$ such that $L(\Lambda_2)q_0$ is amenable relative to $1\otimes L(\Gamma)$. But since $L(\Lambda_2)q_0$ is also amenable relative to $L(\Gamma)\otimes1$, it follows from \cite[Proposition 2.7]{PV11} that $L(\Lambda_2)q_0$ is amenable relative to $\C1$, hence a contradiction.

We conclude that $q=1$ so that $N\prec^f L(\Gamma)\otimes1$. By symmetry, we also get that $N\prec^f 1\otimes L(\Gamma)$. Hence $N\prec^f\C1$, so that $N$ is atomic.

To prove that $L(\Lambda_1 \Lambda_2)' \cap L(\Gamma \times \Gamma)$ is atomic, it suffices to prove that every abelian von Neumann subalgebra $D \subset L(\Lambda_1 \Lambda_2)' \cap L(\Gamma \times \Gamma)$ is atomic. But then $D$ is amenable and its normalizer contains $L(\Lambda_1 \Lambda_2)$, so that $D$ is indeed atomic.
\end{proof}

The proof of the following lemma is essentially contained in the proof of \cite[Proposition 12]{OP03}.

\begin{lemma}\label{lem7-SV}
For every minimal projection $e \in L(\Lambda_1 \Lambda_2)' \cap L(\Gamma \times \Gamma)$, there exist projections $p \in M_n(\C) \ot L(\Gamma)$, $q \in L(\Gamma) \ot M_m(\C)$ and a partial isometry $u \in M_{n,1}(\C) \ot L(\Gamma \times \Gamma) \ot M_{m,1}(\C)$ such that $u^* u = e$, $uu^* = p \ot q$ and such that
\begin{align*}
\text{either}\quad & u L(\Lambda_1) u^* \subset p (M_n(\C) \ot L(\Gamma)) p \ot q & \;\text{and}\quad & u L(\Lambda_2) u^* \subset p \ot q (L(\Gamma) \ot M_m(\C)) q \;\; ,\\
\text{or}\quad & u L(\Lambda_1) u^* \subset p \ot q (L(\Gamma) \ot M_m(\C)) q & \;\text{and}\quad & u L(\Lambda_2) u^* \subset p (M_n(\C) \ot L(\Gamma)) p \ot q \; .
\end{align*}
\end{lemma}

\begin{proof}
By \cite[Proposition 2.7]{PV11}, $L(\Lambda_2)e$ is nonamenable relative to either $L(\Gamma) \ot 1$ or $1 \ot L(\Gamma)$. Assume that $L(\Lambda_2)e$ is nonamenable relative to $L(\Gamma) \ot 1$. By Lemma \ref{lem1}, $L(\Lambda_1)e \prec L(\Gamma) \ot 1$. Take a projection $p \in M_n(\C) \ot L(\Gamma)$, a nonzero partial isometry $v \in (p \ot 1) (M_{n,1}(\C) \ot L(\Gamma \times \Gamma))e$ and a unital normal $*$-homomorphism $\theta : L(\Lambda_1) \recht p (M_n(\C) \ot L(\Gamma))p$ such that
$$(\theta(x) \ot 1)v = v x \quad\text{for all}\;\; x \in L(\Lambda_1) \; .$$
Since $\Gamma \in \Crss$ and $\Lambda_1$ is nonamenable, the relative commutant $\theta(L(\Lambda_1))' \cap p (M_n(\C) \ot L(\Gamma))p$ is atomic. Cutting with a minimal projection, we may assume that this relative commutant equals $\C p$.

Write $P := L(\Lambda_1)' \cap L(\Gamma \times \Gamma)$ and note that $v^* v, e \in P$ with $v^* v \leq e$. Since $L(\Lambda_2) \subset P$, we have that $\cZ(P) \subset L(\Lambda_1 \Lambda_2)' \cap L(\Gamma \times \Gamma)$. It follows that $\cZ(P) e = \C e$. So, $ePe$ is a II$_1$ factor and we can take partial isometries $v_1,\ldots,v_m \in ePe$ with $v_i v_i^* \leq v^* v$ for all $i$ and $\sum_{i=1}^m v_i^* v_i = e$. Define $u \in M_{n,1}(\C) \ot L(\Gamma \times \Gamma) \ot M_{m,1}(\C)$ given by $u = \sum_{i=1}^m v v_i \ot e_{i1}$.

Since $v P v^*$ commutes with $\theta(L(\Lambda_1)) \ot 1$, we have $v P v^* \subset p \ot L(\Gamma)$ and we can define the normal $*$-homomorphism $\eta : v^* v P v^* v \recht L(\Gamma)$ such that $v y v^* = p \ot \eta(y)$ for all $y \in v^* v P v^* v$. By construction, $u^* u = e$ and $u u^* = p \ot q$ where $q \in L(\Gamma) \ot M_m(\C)$ is the projection given by $q = \sum_{i=1}^m \eta(v_i v_i^*) \ot e_{ii}$. Defining the $*$-homomorphism
$$\etatil : e P e \recht q(L(\Gamma) \ot M_m(\C))q : \etatil(y) = \sum_{i,j=1}^m \eta(v_i y v_j^*) \ot e_{ij}$$
and using that $L(\Lambda_2)e \subset ePe$, we get that
$$u L(\Lambda_1) u^* = \theta(L(\Lambda_1)) \ot q \quad\text{and}\quad u L(\Lambda_2) u^* = p \ot \etatil(L(\Lambda_2) e) \; .$$
This concludes the proof of the lemma.
\end{proof}

Recall from Section \ref{height} the notion of height of an element in a group von Neumann algebra (here, $L(\Gamma \times \Gamma)$), as well as the height of a subgroup of $\cU(L(\Gamma \times \Gamma))$. The proof of the following lemma is very similar to the proof of \cite[Theorem 4.1]{Io10}.

\begin{lemma}\label{lem8-SV}
For every projection $p \in L(\Lambda_1 \Lambda_2)' \cap L(\Gamma \times \Gamma)$, we have that $h_{\Gamma \times \Gamma}(\Lambda_1 \Lambda_2 p) > 0$.
\end{lemma}

\begin{proof}
Fix a minimal projection $p \in L(\Lambda_1 \Lambda_2)' \cap L(\Gamma \times \Gamma)$. It suffices to prove that $h_{\Gamma \times \Gamma}(\Lambda_1 \Lambda_2 p) > 0$. Using the conjugacy of Lemma \ref{lem7-SV}, we see that the heights of $\Lambda_1 p$ and $\Lambda_2 p$ do not interact, so that it suffices to prove that $h_{\Gamma \times \Gamma}(\Lambda_i p) > 0$ for $i=1,2$. By symmetry, it is enough to prove this for $i=1$.

Assume for contradiction that $h_{\Gamma \times \Gamma}(\Lambda_1 p) = 0$. Take a sequence $v_n \in \Lambda_1$ such that $h_{\Gamma \times \Gamma}(v_n p) \recht 0$. For every finite subset $S \subset \Gamma \times \Gamma$, we denote by $P_S$ the orthogonal projection of $L^2(M)$ onto the linear span of $L^2(A) u_g$, $g \in S$. We claim that for every sequence of unitaries $w_n \in L(\Gamma \times \Gamma)$, every $a \in M \ominus L(\Gamma \times \Gamma)$ and every finite subset $S \subset \Gamma \times \Gamma$, we have that
$$\lim_n \|P_S(p v_n a w_n)\|_2 = 0 \; .$$
Since $P_S(x) = \sum_{g \in S} E_A(x u_g^*) u_g$, it suffices to prove that $\|E_A(p v_n a w_n)\|_2 \recht 0$ for all $a \in M \ominus L(\Gamma \times \Gamma)$. Every such $a$ can be approximated by a linear combination of elements of the form $a_0 u_g$ with $a_0 \in A \ominus \C 1$ and $g \in \Gamma \times \Gamma$. So, we may assume that $a \in A \ominus \C 1$. Such an element $a$ can be approximated by a linear combination of elementary tensors, so that we may assume that $a = \bigotimes_{i \in \cG} a_i$ for some finite nonempty subset $\cG \subset \Gamma$ and elements $a_i \in A_0 \ominus \C 1$ with $\|a\| \leq 1$. Note that $\sigma_g(a) \perp \sigma_h(a)$ whenever $g,h \in \Gamma \times \Gamma$ and $g \cdot \cG \neq h \cdot \cG$ (where we use the left right action of $\Gamma \times \Gamma$ on $\Gamma$).

Choose $\eps > 0$. By Lemma \ref{lem7-SV}, we can take a finite subset $F_0 \subset \Gamma$ such that, writing $F = \Gamma \times F_0 \cup F_0 \times \Gamma$, we have $\|p v - P_F(p v)\|_2 \leq \eps$ for all $v \in \Lambda_1$. Then,
$$\|E_A(p v_n a w_n) - E_A(P_F(p v_n) a w_n)\|_2 \leq \eps$$
for all $n$, so that in order to prove the claim, it suffices to prove that $\|E_A(P_F(p v_n) a w_n)\|_2 \recht 0$. Put $\kappa = 2 |F_0| |\cG|^2$. Note that for every $h \in \Gamma \times \Gamma$, the set $\{g \in F \mid g \cdot \cG = h \cdot \cG \}$ contains at most $\kappa$ elements. Using the Fourier decomposition for elements in $L(\Gamma \times \Gamma)$, we have
$$E_A(P_F(p v_n) a w_n) = \sum_{g \in F} (p v_n)_g \, (w_n)_{g^{-1}} \, \sigma_g(a) \; .$$
Thus, for all $h \in \Gamma \times \Gamma$, we have
$$|\langle E_A(P_F(p v_n) a w_n) , \sigma_h(a) \rangle| \leq \kappa \, h_{\Gamma \times \Gamma}(p v_n) \; .$$
But then, using the Cauchy-Schwarz inequality, we get that
\begin{align*}
\|E_A(P_F(p v_n) a w_n)\|_2^2 & \leq \sum_{h \in F} |\langle E_A(P_F(p v_n) a w_n) \, , \, (p v_n)_h \, (w_n)_{h^{-1}} \, \sigma_h(a) \rangle| \\
& \leq \kappa \, h_{\Gamma \times \Gamma}(p v_n) \, \sum_{h \in \Gamma \times \Gamma} |(p v_n)_h| \, |(w_n)_{h^{-1}}| \\
& \leq \kappa \, h_{\Gamma \times \Gamma}(p v_n) \, \|p v_n\|_2 \, \|w_n\|_2 \leq \kappa \, h_{\Gamma \times \Gamma}(p v_n) \recht 0 \; .
\end{align*}
So, the claim is proved.

Put $\delta = \|p\|_2 / 4$. Because $\Gamma \in \Crss$ and $B \subset M$ is a Cartan subalgebra, we have that $B \prec^f A$. By Lemma \ref{lem3}, we have $B \not\prec L(\Gamma \times \Gamma)$ and we can take a unitary $b \in \cU(B)$ such that $\|E_{L(\Gamma \times \Gamma)}(b)\|_2 \leq \delta$. Since $B \prec^f A$, we can take a finite subset $S \subset \Gamma \times \Gamma$ such that $\| p d - P_S(p d)\|_2 \leq \delta$ for all $d \in \cU(B)$. For every $n$, we have that $v_n b v_n^* \in \cU(B)$. Therefore,
$$\|p v_n b v_n^* - P_S(p v_n b v_n^*)\|_2 \leq \delta$$
for all $n$. Since $\|E_{L(\Gamma \times \Gamma)}(b)\|_2 \leq \delta$, writing $b_1 = b - E_{L(\Gamma \times \Gamma)}(b)$, we get
$$\| P_S(p v_n b v_n^*) - P_S(p v_n b_1 v_n^*) \|_2 \leq \delta \; .$$
By the claim above, we can fix $n$ large enough such that $\|P_S(p v_n b_1 v_n^*) \|_2 \leq \delta$. So, we have proved that
$$\|p\|_2 = \|p v_n b v_n^*\|_2 \leq 3\delta < \|p\|_2 \; ,$$
which is absurd. So, we have shown that $h_{\Gamma \times \Gamma}(p \Lambda_1) > 0$ and the lemma is proved.
\end{proof}

\begin{lemma}\label{lem9-SV}
There exists a unitary $u \in L(\Gamma \times \Gamma)$ such that $u \Lambda_1 \Lambda_2 u^* \subset \T (\Gamma \times \Gamma)$. Also, the unitary representation $\{\Ad v\}_{v \in \Lambda_1 \Lambda_2}$ is weakly mixing on $L^2(M) \ominus \C 1$.
\end{lemma}
\begin{proof}
Write $\Lambda_0 = \Lambda_1 \Lambda_2$. Denote the action of $\Lambda$ on $B$ by $\gamma_v(b) = vbv^*$ for all $v \in \Lambda$, $b \in B$. Define $K < \Lambda$ as the virtual centralizer of $\Lambda_0$ inside $\Lambda$, i.e.\ $K$ consists of all $v \in \Lambda$ such that the set $\{w v w^{-1} \mid w \in \Lambda_0\}$ is finite. Define $B_0 \subset B$ as the von Neumann algebra generated by the unital $*$-algebra consisting of all $b \in B$ such that $\{\gamma_v(b) \mid v \in \Lambda_0\}$ spans a finite dimensional subspace of $B$. Note that $B_0$ is globally invariant under $\gamma_v$, $v \in \Lambda_0$. Viewing $M$ as the crossed product $M = B \rtimes \Lambda$, we have by construction that the unitary representation $\{\Ad v\}_{v \in \Lambda_0}$ is weakly mixing on $L^2(M) \ominus L^2(B_0 \rtimes K)$.

For every $g \in \Gamma$, define $\stab g$ as in the beginning of the proof of Lemma \ref{lem1-SV}. We have $L(\Lambda_0) \subset L(\Gamma \times \Gamma)$ and $L(\Lambda_0) \not\prec L(\stab g)$ for all $g \in \Gamma$. By Lemma \ref{lem7}, we have $B_0 \rtimes K \subset L(\Gamma \times \Gamma)$. Since we can take decreasing sequences of finite index subgroups $\Lambda_{i,n} < \Lambda_i$, $i=1,2$, such that $K = \bigcup_n C_\Lambda(\Lambda_{1,n} \Lambda_{2,n})$, it follows from Lemma \ref{lem6} that $L(K)$ is contained in a hyperfinite von Neumann algebra. So, $K$ is amenable and thus also $B_0 \rtimes K$ is amenable. Since $B_0 \rtimes K$ is normalized by $\Lambda_0$, it follows from Lemma \ref{lem6} that $B_0 \rtimes K$ is atomic. So, $K$ is a finite group and $B_0$ is atomic. We can then take a minimal projection $p \in B_0 \rtimes K$ and finite index subgroups $\Lambda_3 < \Lambda_1$ and $\Lambda_4 < \Lambda_2$ such that $p$ commutes with $\Lambda_3 \Lambda_4$.

Lemmas \ref{lem6}, \ref{lem7-SV} and \ref{lem8-SV} apply to the commuting nonamenable subgroups $\Lambda_3,\Lambda_4 < \Lambda$. So, by Lemma \ref{lem8-SV}, we get that $h_{\Gamma \times \Gamma}(p \Lambda_3 \Lambda_4) > 0$. By construction, the unitary representation $\{\Ad v\}_{v \in \Lambda_3 \Lambda_4}$ is weakly mixing on $p L(\Gamma \times \Gamma) p \ominus \C p$. For every $g \in \Gamma \times \Gamma$ with $g \neq e$, the centralizer $C_{\Gamma \times \Gamma}(g)$ is either amenable or of the form $\Gamma \times L$ or $L \times \Gamma$ with $L < \Gamma$ amenable. Therefore, $L(\Lambda_3 \Lambda_4) \not\prec L(C_{\Gamma \times \Gamma}(g))$ for all $g \neq e$. It then first follows from Theorem \ref{thm.group-vnalg} that $p = 1$, so that we could have taken $\Lambda_3 = \Lambda_1$ and $\Lambda_4 = \Lambda_2$, and then also that there exists a unitary $u \in L(\Gamma \times \Gamma)$ such that $u \Lambda_1 \Lambda_2 u^* \subset \T (\Gamma \times \Gamma)$.

Since we also proved that $B_0 \rtimes K = \C 1$, it follows as well that the unitary representation $\{\Ad v\}_{v \in \Lambda_1 \Lambda_2}$ is weakly mixing on $L^2(M) \ominus \C 1$.
\end{proof}

\begin{lemma}\label{lem10-a-SV}
Whenever $\Lambda_2 \subset \T (\{e\} \times \Gamma)$ is a nonamenable subgroup, we have $M \cap \Lambda_2' = L(\Gamma) \ot 1$.
\end{lemma}
\begin{proof}
Define $\Gamma_2 < \Gamma$ such that $\T \Lambda_2 = \T (\{e\} \times \Gamma_2)$. Then $\Gamma_2$ is nonamenable and $M \cap \Lambda_2' = M \cap L(\{e\} \times \Gamma_2)'$. Since $\Gamma_2 < \Gamma$ is relatively icc, we have $M \cap L(\{e\} \times \Gamma_2)' \subset A \rtimes (\Gamma \times \{e\})$. Since the action $\{e\} \times \Gamma \actson A$ is mixing, it follows that $M \cap L(\{e\} \times \Gamma_2)' \subset L(\Gamma \times \{e\})$. So, $M \cap \Lambda_2' \subset L(\Gamma) \ot 1$ and the converse inclusion is obvious.
\end{proof}

\begin{lemma}\label{lem10-SV}
There exist commuting subgroups $H_1,H_2 < \Lambda$ and a unitary $u \in M$ such that $\Lambda_i < H_i$ for $i=1,2$ and $u \T H_1 H_2 u^* = \T(\Gamma \times \Gamma)$.
\end{lemma}

\begin{proof}
By Lemma \ref{lem9-SV}, after a unitary conjugacy, we may assume that $\Lambda_1,\Lambda_2 < \Lambda$ are commuting nonamenable subgroups with $\Lambda_1 \Lambda_2 \subset \T(\Gamma \times \Gamma)$. Since $\Gamma$ is torsion free and belongs to $\Crss$, we have that $C_\Gamma(g)$ is amenable for every $g \neq e$. Therefore, after exchanging $\Lambda_1$ and $\Lambda_2$ if needed, we have $\Lambda_1 \subset \T (\Gamma \times \{e\})$ and $\Lambda_2 \subset \T(\{e\} \times \Gamma)$.

Denote by $\{\gamma_v\}_{v \in \Lambda}$ the action of $\Lambda$ on $B$. Define $H_1 < \Lambda$ as the virtual centralizer of $\Lambda_2$ inside $\Lambda$. So, $H_1$ consists of all $v \in \Lambda$ that commute with a finite index subgroup of $\Lambda_2$. Similarly, define $B_1$ as the von Neumann algebra generated by the $*$-algebra consisting of all $b \in B$ such that $\gamma_v(b)=b$ for all $v$ in a finite index subgroup of $\Lambda_2$. Since finite index subgroups of $\Lambda_2$ are nonamenable, it follows from Lemma \ref{lem10-a-SV} that $B_1 \rtimes H_1 \subset L(\Gamma) \ot 1$. We also find that
$$L(\Gamma) \ot 1 \subset L(\Lambda_2)' \cap (B \rtimes \Lambda) \subset B_1 \rtimes H_1 \; .$$
So, $B_1 \rtimes H_1 = L(\Gamma) \ot 1$. In particular, the subgroups $H_1,\Lambda_2 < \Lambda$ commute. Because $\Gamma \in \Crss$ and $B_1 \subset L(\Gamma) \ot 1$ is normalized by $H_1$, it follows that $B_1$ is atomic. Since $\Lambda_1 < H_1$, the unitaries $v \in \Lambda_1 \Lambda_2$ normalize $B_1$. By Lemma \ref{lem9-SV}, they induce a weakly mixing action on $B_1$. Since $B_1$ is atomic, this forces $B_1 = \C 1$. We conclude that $L(H_1) = L(\Gamma) \ot 1$.

We now apply Lemmas \ref{lem6}, \ref{lem7-SV} and \ref{lem8-SV} to the commuting nonamenable subgroups $H_1,\Lambda_2 < \Lambda$. We conclude that $h_\Gamma(H_1) > 0$. Since $L(H_1) = L(\Gamma) \ot 1$, the group $H_1$ is icc. So, the action $\{\Ad v\}_{v \in H_1}$ on $L(\Gamma) \ominus \C 1$ is weakly mixing. Since for $g \neq e$, the group $C_\Gamma(g)$ is amenable, also $L(H_1) \not\prec L(C_\Gamma(g))$. So, by Theorem \ref{thm.group-vnalg}, there exists a unitary $u_1 \in L(\Gamma)$ such that $(u_1 \ot 1) H_1 (u_1^* \ot 1) = \T(\Gamma \times \{e\})$.

Applying the same reasoning as above to the virtual centralizer of $H_1$ inside $\Lambda$, we find a subgroup $H_2 < \Lambda$, containing $\Lambda_2$ and commuting with $H_1$, and we find a unitary $u_2 \in L(\Gamma)$ such that $(1 \ot u_2) H_2 (1 \ot u_2^*) \subset \T (\{e\} \rtimes \Gamma)$. So, we get that
$$(u_1 \ot u_2) H_1 H_2 (u_1^* \ot u_2^*) = \T(\Gamma \times \Gamma) \; .$$
\end{proof}

Finally, we are ready to prove Theorem \ref{main-SV}. As mentioned above, Theorem \ref{thmB} is a direct consequence of Theorem \ref{main-SV}.

\begin{proof}[Proof of Theorem \ref{main-SV}]
First assume that $(B_0,\Lambda_0)$ is a gms decomposition of $A_0$ satisfying $\beta_g(B_0) = B_0$ and $\beta_g(\Lambda_0) = \Lambda_0$ for all $g \in \Gamma$. Then, $\{\beta_g\}_{g \in \Gamma}$ defines an action of $\Gamma$ by automorphisms of the group $\Lambda_0$. We can co-induce this to the action of $\Gamma \times \Gamma$ by automorphisms of the direct sum group $\Lambda_0^{(\Gamma)}$ given by $(g,h) \cdot \pi_k(v) = \pi_{gkh^{-1}}(\beta_h(v))$ for all $g,h,k \in \Gamma$, $v \in \Lambda_0$, where $\pi_k : \Lambda_0 \recht \Lambda_0^{(\Gamma)}$ denotes the embedding as the $k$'th direct summand. Putting $B = B_0^\Lambda$ and $\Lambda := \Lambda_0^{(\Gamma)}\rtimes(\Gamma\times\Gamma)$, we have found the crossed product decomposition $M = B \rtimes \Lambda$. It is easy to check that $B \subset M$ is maximal abelian, so that $(B,\Lambda)$ is indeed a gms decomposition of $M$.

Conversely, assume that $(B,\Lambda)$ is an arbitrary gms decomposition of $M$. By Lemma \ref{lem10-SV} and after a unitary conjugacy, we have $\Gamma \times \Gamma \subset \T \Lambda$. Denoting by $\Delta : M \recht M \ovt M$ the dual coaction associated with $(B,\Lambda)$ and given by $\Delta(b) = b \ot 1$ for all $b \in B$ and $\Delta(v) = v \ot v$ for all $v \in \Lambda$, this means that $\Delta(u_{(g,h)})$ is a multiple of $u_{(g,h)} \ot u_{(g,h)}$ for all $(g,h) \in \Gamma \times \Gamma$.

Denote $A_{0,e}:=\pi_e(A_0)\subset A$ and observe that $A_{0,e}$ commutes with all $u_{(g,g)}$, $g \in \Ker \beta$. Then, $\Delta(A_{0,e})$ commutes with all $u_{(g,g)} \ot u_{(g,g)}$, $g \in \Ker \beta$. Since $\Gamma$ is a torsion free group in $\Crss$, the nontrivial normal subgroup $\Ker \be < \Gamma$ must be nonamenable and thus relatively icc. It follows that the unitary representation $\{\Ad u_{(g,g)}\}_{g\in\Ker\beta}$ is weakly mixing on $L^2(M)\ominus L^2(A_{0,e})$. This implies that $\Delta(A_{0,e})\subset A_{0,e}\tensor A_{0,e}$.

By Lemma \ref{lem4}, we get a crossed product decomposition $A_0 = B_0 \rtimes \Lambda_0$ such that $\pi_e(B_0) = B \cap A_{0,e}$ and $\pi_e(\Lambda_0) = \Lambda \cap A_{0,e}$. For every $g \in \Gamma$, we have that $u_{(g,g)} \in \T \Lambda$. So, $u_{(g,g)}$ normalizes both $B$ and $A_{0,e}$, so that $\beta_g(B_0) = B_0$. Also, $u_{(g,g)}$ normalizes both $\Lambda$ and $A_{0,e}$, so that $\beta_g(\Lambda_0) = \Lambda_0$. For every $g \in \Gamma$, we have that $u_{(g,e)} \in \T \Lambda$ so that $u_{(g,e)}$ normalizes $B$ and $\Lambda$. It follows that $\pi_g(B_0) \subset B$ and $\pi_g(\Lambda_0) \subset \Lambda$ for all $g \in \Gamma$. We conclude that
\begin{equation}\label{eq.inclusion}
B_0^\Gamma \subset B \quad\text{and}\quad \Lambda_0^{(\Gamma)} \rtimes (\Gamma \times \Gamma) \subset \T \Lambda \; .
\end{equation}
Since $A_0$ is generated by $B_0$ and $\Lambda_0$, we get that $M$ is generated by $B_0^\Gamma$ and $\Lambda_0^{(\Gamma)} \rtimes (\Gamma \times \Gamma)$. Since $M$ is also the crossed product of $B$ and $\Lambda$, it follows from \eqref{eq.inclusion} that $B_0^\Gamma = B$ and $\T \Lambda_0^{(\Gamma)} \rtimes (\Gamma \times \Gamma) = \T \Lambda$. In particular, $B_0 \subset A_0$ must be maximal abelian. So, $(B_0,\Lambda_0)$ is a gms decomposition of $A_0$ that is $\{\beta_g\}_{g \in \Gamma}$-invariant, while the gms decomposition $(B,\Lambda)$ of $M$ is unitarily conjugate to the gms decomposition associated with $(B_0,\Lambda_0)$.

It remains to prove statements \ref{main.uc} and \ref{main.autc}. Take $\{\beta_g\}_{g \in \Gamma}$-invariant gms decompositions $(B_0,\Lambda_0)$ and $(B_1,\Lambda_1)$ of $A_0$. Denote by $(B,\Lambda)$ and $(B',\Lambda')$ the associated gms decompositions of $M$.

To prove \ref{main.uc}, assume that $u \in M$ is a unitary satisfying $uBu^\ast=B'$ and $u\T\Lambda u^\ast=\T\Lambda'$. It follows that for all $g\in\Gamma\times\Gamma$, we have $uu_gu^\ast\in \cU(A)u_{\varphi(g)}$ where $\varphi \in \Aut(\Gamma\times\Gamma)$. Write $u=\sum_{h\in\Gamma\times\Gamma}a_hu_h$ with $a_h\in A$ for the Fourier decomposition of $u$. It follows that $\{\varphi(g)^{-1}hg\mid g\in\Gamma\times\Gamma\}$ is a finite set whenever $a_h\neq0$. Since $\Gamma\times\Gamma$ is icc, it follows that $a_h$ can only be nonzero for one $h\in\Gamma\times\Gamma$. So $u$ is of the form $u=a_hu_h$. Since $u_h$ normalizes both $B$ and $\Lambda$, we may replace $u$ with $uu_h^\ast$ so that $u\in\mathcal{U}(A)$.

For each $g\in\Gamma$, we define $E_g\colon A\to A_0$ by $E_g(x)=\pi_g^{-1}(E_{\pi_g(A_0)}(x))$, $x\in A$. Let $(g_n)_{n\in\N}$ be a sequence in $\Gamma$ that tends to infinity, and let $b\in B_0$. Since $(\pi_{g_n}(b))_{n\in\N}$ is an asymptotically central sequence in $A$, we get that
$$ B_1\ni E_{g_n}(u\pi_{g_n}(b)u^\ast)\to b \; , $$
hence $B_0\subset B_1$. By symmetry, it follows that $B_0=B_1$. Similarly, we see that $\T\Lambda_0=\T\Lambda_1$ so we conclude that $(B_0,\Lambda_0)$ and $(B_1,\Lambda_1)$ are identical gms decompositions of $A_0$.

To prove \ref{main.autc}, assume that $\theta \in \Aut(M)$ is an automorphism satisfying $\theta(B) = B'$ and $\theta(\T\Lambda) = \T\Lambda'$. Define the commuting subgroups $\Lambda_1,\Lambda_2 < \Lambda'$ such that $\theta(\T(\Gamma \times \{e\})) = \T \Lambda_1$ and $\theta(\T(\{e \} \times \Gamma)) = \T \Lambda_2$. Applying Lemma \ref{lem10-SV} to the gms decomposition $(B',\Lambda')$ of $M$ and these commuting subgroups $\Lambda_1,\Lambda_2 < \Lambda'$, we find commuting subgroups $H_1,H_2 < \Lambda'$ and a unitary $u \in M$ such that $\Lambda_i < H_i$ for $i=1,2$ and $u \T H_1 H_2 u^* = \T (\Gamma \times \Gamma)$. Since $\Gamma \times \{e\}$ and $\{e\} \times \Gamma$ are each other's centralizer inside $\Lambda$ and since $\theta(\T \Lambda) = \T \Lambda'$, we must have that $\Lambda_i = H_i$ for $i=1,2$.

Writing $\theta_1 = \Ad u \circ \theta$, we have proved that $\theta_1(\T(\Gamma \times \Gamma)) = \T (\Gamma \times \Gamma)$. This equality induces an automorphism of $\Gamma \times \Gamma$. Since $\Gamma$ is a torsion free group in $\Crss$, all automorphisms of $\Gamma \times \Gamma$ are either of the form $(g,h) \mapsto (\vphi_1(g),\vphi_2(h))$ or of the form $(g,h) \mapsto (\vphi_1(h),\vphi_2(g))$ for some automorphisms $\vphi_i \in \Aut(\Gamma)$. The formulas $\zeta(u_{(g,h)}) = u_{(h,g)}$ and $\zeta(\pi_k(a)) = \pi_{k^{-1}}(\beta_k(a))$ for all $g,h,k \in \Gamma$ and $a \in A_0$ define an automorphism $\zeta \in \Aut(M)$ satisfying $\zeta(B') = B'$ and $\zeta(\Lambda') = \Lambda'$. So composing $\theta$ with $\zeta$ if necessary, we may assume that we have $\vphi_1,\vphi_2 \in \Aut(\Gamma)$ such that $\theta_1(u_{(g,h)}) \in \T u_{(\vphi_1(g),\vphi_2(h))}$ for all $g,h \in \Gamma$. We still have that the gms decompositions $(\theta_1(B),\theta_1(\Lambda))$ and $(B',\Lambda')$ of $M$ are unitarily conjugate.

Because $u_{(g,g)}$ commutes with $\pi_e(A_0)$ for all $g \in \Ker \be$, the unitary representation on $L^2(M) \ominus \C 1$ given by $\{\Ad u_{(\vphi_1(g),\vphi_2(g))}\}_{g \in \Ker \be}$  is not weakly mixing. There thus exists a $k \in \Gamma$ such that $\vphi_1(g) k = k \vphi_2(g)$ for all $g$ in a finite index subgroup of $\Ker \be$. So after replacing $\theta$ by $(\Ad u_{(e,k)}) \circ \theta$, which globally preserves $B'$ and $\Lambda'$, we may assume that $\vphi_1(g) = \vphi_2(g)$ for all $g$ in a finite index subgroup of $\Ker \be$. Since $\Ker \be < \Gamma$ is relatively icc, this implies that $\vphi_1(g) = \vphi_2(g)$ for all $g \in \Ker \be$. Since $\Ker \be$ is a normal subgroup of $\Gamma$, it follows that $\vphi_1(k) \vphi_2(k)^{-1}$ commutes with $\vphi_1(g)$ for all $k \in \Gamma$ and $g \in \Ker \be$. So, $\vphi_1 = \vphi_2$ and we denote this automorphism as $\vphi$.

Taking the commutant of the unitaries $u_{(g,g)}$, $g \in \Ker \be$, it follows that $\theta_1(\pi_e(A_0)) = \pi_e(A_0)$. We define the automorphism $\theta_0 \in \Aut(A_0)$ such that $\theta_1 \circ \pi_e = \pi_e \circ \theta_0$. Since $\theta_1(u_{(g,g)}) \in \T u_{(\vphi(g),\vphi(g))}$, we get that $\theta_0 \circ \be_g = \be_{\vphi(g)} \circ \theta_0$ for all $g \in \Gamma$. It follows that $(\theta_0(B_0),\theta_0(\Lambda_0))$ is a $\{\be_g\}_{g \in \Gamma}$-invariant gms decomposition of $A_0$. The associated gms decomposition of $M$ is $(\theta_1(B),\theta_1(\Lambda))$. This gms decomposition of $M$ is unitarily conjugate with the gms decomposition $(B',\Lambda')$. It then follows from \ref{main.uc} that $\theta_0(B_0) = B_1$ and $\theta_0(\T \Lambda_0) = \T \Lambda_1$.
\end{proof}

\begin{prop}\label{prop-final}
Under the same hypotheses and with the same notations as in Theorem \ref{main-SV}, if $(B_0,\Lambda_0)$ and $(B_1,\Lambda_1)$ are $\{\be_g\}_{g \in \Gamma}$-invariant gms decompositions of $A_0$, then the associated Cartan subalgebras of $M$ given by $B_0^\Gamma$ and $B_1^\Gamma$ are
\begin{enumlist}
\item\label{final.one} unitarily conjugate iff $B_0 = B_1$~;
\item\label{final.two} conjugate by an automorphism of $M$ iff there exists a trace preserving automorphism $\theta_0 : A_0 \recht A_0$ and an automorphism $\vphi \in \Aut(\Gamma)$ such that $\theta_0(B_0) = B_1$ and $\theta_0 \circ \be_g = \be_{\vphi(g)} \circ \theta_0$ for all $g \in \Gamma$.
\end{enumlist}
\end{prop}

\begin{proof}
To prove \ref{final.one}, it suffices to prove that $B_0^\Gamma \not\prec B_1^\Gamma$ if $B_0 \neq B_1$. Take a unitary $u\in \mathcal{U}(B_0)$ such that $u \not\in B_1$.
Then $\|E_{B_1}(u)\|_2<1$. Let $\{g_1,g_2,\ldots\}$ be an enumeration of $\Gamma$ and define the sequence of unitaries $(w_n)\subset\mathcal{U}(B_0^{\Gamma})$ by $w_n= \pi_{g_{n+1}}(u) \, \pi_{g_{n+2}}(u) \, \cdots \, \pi_{g_{2n}}(u)$. Then $\|E_{B_1^{\Gamma}}(xw_ny)\|_2\to0$ for all $x,y\in M$ so that $B_0^{\Gamma}\nprec B_1^{\Gamma}$.

To prove \ref{final.two}, denote by $(B,\Lambda)$ and $(B',\Lambda')$ the gms decompositions of $M$ associated with $(B_0,\Lambda_0)$ and $(B_1,\Lambda_1)$. Assume that $\theta \in \Aut(M)$ satisfies $\theta(B) = B'$. Then, $(B',\theta(\Lambda))$ is a gms decomposition of $M$. By Theorem \ref{main-SV}, $(B',\theta(\Lambda))$ is unitarily conjugate with the gms decomposition associated with a $\{\beta_g\}_{g \in \Gamma}$-invariant gms decomposition $(B_2,\Lambda_2)$ of $A_0$. By \ref{final.one}, we must have $B_2 = B_1$. So the gms decompositions associated with $(B_0,\Lambda_0)$ and $(B_1,\Lambda_2)$ are conjugate by an automorphism of $M$. By Theorem \ref{main-SV}\ref{main.autc} there exists an automorphism $\theta_0 \in \Aut(A_0)$ as in \ref{final.two}.
\end{proof}

\section{\boldmath Examples of II$_1$ factors with a prescribed number of group measure space decompositions}\label{sec.examples}

For every amenable tracial von Neumann algebra $(A_0,\tau_0)$ and for every trace preserving action of $\Gamma = \F_\infty$ on $(A_0,\tau_0)$ with nontrivial kernel, Theorem \ref{main-SV} gives a complete description of all gms decompositions of the II$_1$ factor $M = A_0^{\Gamma} \rtimes (\Gamma \times \Gamma)$ in terms of the $\Gamma$-invariant gms decompositions of $A_0$.

In this section, we construct a family of examples where these $\Gamma$-invariant gms decompositions of $A_0$ can be explicitly determined. In particular, this gives a proof of Theorem \ref{thmA}. We will construct $A_0$ of the form $A_0 = L^\infty(K) \rtimes H_1$ where $H_1$ is a countable abelian group and $H_1 \hookrightarrow K$ is an embedding of $H_1$ as a dense subgroup of the compact second countable group $K$. Note that we can equally view $K$ as $\widehat{H_2}$ where $H_2$ is a countable abelian group and the embedding $H_1 \hookrightarrow \widehat{H_2}$ is given by a bicharacter $\Om : H_1 \times H_2 \recht \T$ that is nondegenerate: if $g \in H_1$ and $\Om(g,h) = 1$ for all $h \in H_2$, then $g = e$~; if $h \in H_2$ and $\Om(g,h)=1$ for all $g \in H_1$, then $h = e$.

We can then view $L^\infty(\widehat{H_2}) \rtimes H_1$ as being generated by the group von Neumann algebras $L(H_1)$ and $L(H_2)$, with canonical unitaries $\{u_g\}_{g \in H_1}$ and $\{u_h\}_{h \in H_2}$, satisfying $u_g u_h = \Om(g,h) u_h u_g$ for all $g \in H_1$, $h \in H_2$.

We call a direct sum decomposition $H_1 = S_1 \oplus T_1$ \emph{admissible} if the closures of $S_1$, $T_1$ in $\widehat{H_2}$ give a direct sum decomposition of $\widehat{H_2}$. This is equivalent to saying that $H_2 = S_2 \oplus T_2$ in such a way that $S_1 = S_2^\perp$, $S_2 = S_1^\perp$, $T_2 = T_1^\perp$ and $T_1 = T_2^\perp$, where the ``orthogonal complement'' is defined w.r.t.\ $\Om$.

\begin{prop}\label{prop.compute-invariant-gms-decompositions}
Let $L_1, L_2$ be torsion free abelian groups and $L_1 \hookrightarrow \widehat{L_2}$ a dense embedding. Put $\Gamma_0 = \SL(3,\Z)$ and $H_i = L_i^3$. Consider the natural action of $\Gamma_0$ on the direct sum embedding $H_1 \hookrightarrow \widehat{H_2}$, defining the trace preserving action $\{\be_g\}_{g \in \Gamma_0}$ of $\Gamma_0$ on $A_0 = L^\infty(\widehat{H_2}) \rtimes H_1$.

Whenever $L_1 = P_1 \oplus Q_1$ is an admissible direct sum decomposition with corresponding $L_2 = P_2 \oplus Q_2$, put $S_i = P_i^3$, $T_i = Q_i^3$ and define $B_0 = L(S_1) \vee L(S_2)$, $\Lambda_0 = T_1 T_2$.

Then $(B_0,\Lambda_0)$ is a $\{\be_g\}_{g \in \Gamma_0}$-invariant gms decomposition of $A_0$. Every $\{\be_g\}_{g \in \Gamma_0}$-invariant gms decomposition of $A_0$ is of this form for a unique admissible direct sum decomposition $L_1 = P_1 \oplus Q_1$.
\end{prop}

\begin{proof}

Let $(B_0,\Lambda_0)$ be a $\{\be_g\}_{g \in \Gamma_0}$-invariant gms decomposition of $A_0$. Define the subgroup $\Gamma_1 < \Gamma_0$ as
$$\Gamma_1 = \Gamma_0 \cap \begin{pmatrix} 1 & * & * \\ 0 & * & * \\ 0 & * & * \end{pmatrix}$$
We also put $H_1^{(1)} = L_1 \oplus 0 \oplus 0$. Because $L_1$ is torsion free, the following holds.
\begin{itemlist}
\item $a \cdot g = g$ for all $a \in \Gamma_1$ and $g \in H_1^{(1)}$.
\item $\Gamma_1 \cdot g$ is infinite for all $g \in H_1 \setminus H_1^{(1)}$.
\item $\Gamma_1^T \cdot h$ is infinite for all $h \in H_2 \setminus \{0\}$, where $\Gamma_1^T$ denotes the transpose of $\Gamma_1$.
\end{itemlist}
From these observations, it follows that $L(H_1^{(1)})$ is equal to the algebra of $\Gamma_1$-invariant elements in $A_0$ and that $L(H_1^{(1)})$ is also equal to the algebra of elements in $A_0$ that are fixed by some finite index subgroup of $\Gamma_1$. Since both $B_0$ and $\Lambda_0$ are globally $\Gamma_0$-invariant, it follows that $L(H_1^{(1)}) = B_0^{(1)} \rtimes \Lambda_0^{(1)}$ for some von Neumann subalgebra $B_0^{(1)}$ and subgroup $\Lambda_0^{(1)} < \Lambda_0$.

We similarly consider $H_1^{(2)} = 0 \oplus L_1 \oplus 0$ and $H_1^{(3)} = 0 \oplus 0 \oplus L_1$. We conclude that $L(H_1^{(i)}) = B_0^{(i)} \rtimes \Lambda_0^{(i)}$ for all $i=1,2,3$. The subgroups $H_1^{(1)}, H_1^{(2)}$ and $H_1^{(3)}$ generate $H_1$ and $H_1$ is abelian. So, ``everything'' commutes and we conclude that $L(H_1) = B_1 \rtimes \Lambda_1$ for some von Neumann subalgebra $B_1 \subset B_0$ and subgroup $\Lambda_1 < \Lambda_0$.

A similar reasoning applies to $L^\infty(\widehat{H_2}) = L(H_2)$ and we get that $L(H_2) = B_2 \rtimes \Lambda_2$ for $B_2 \subset B$ and $\Lambda_2 < \Lambda_0$.

Since $L(H_1) L(H_2)$ is $\|\,\cdot\,\|_2$-dense in $A_0$ and $L(H_1) \cap L(H_2) = \C 1$, we get that $\Lambda_1 \Lambda_2 = \Lambda_0 = \Lambda_2 \Lambda_1$ and $\Lambda_1 \cap \Lambda_2 = \{e\}$. It then follows that for all $b_i \in B_i$ and $s_i \in \Lambda_i$, $i=1,2$, we have that $E_{B_0}(b_1 v_{s_1} b_2 v_{s_2})$ equals zero unless $s_1 = e$ and $s_2 = e$, in which case, we get $b_1 b_2$. We conclude that $B_1 B_2$ is $\|\,\cdot\,\|_2$-dense in $B_0$.

For every $x \in L(H_i)$ and $g \in H_i$, we denote by $(x)_g = \tau(x u_g^*)$ the $g$-th Fourier coefficient of $x$. Comparing Fourier decompositions, we get for all $x_i \in L(H_i)$ that
\begin{equation}\label{eq.criterion-commute}
x_1 x_2 = x_2 x_1 \;\;\text{iff}\;\; \Om(g,h) = 1 \;\;\text{whenever}\;\; g \in H_1 , h \in H_2 , (x_1)_g \neq 0 , (x_2)_h \neq 0 \; .
\end{equation}
Since $B_i \subset L(H_i)$ and since $B_1$, $B_2$ commute, we obtain from \eqref{eq.criterion-commute} subgroups $S_i \subset H_i$ such that $\Om(g,h) = 1$ for all $g \in S_1$, $h \in S_2$ and such that $B_i \subset L(S_i)$. Since $B_1 B_2$ is dense in $B_0$, it follows that $B_0 \subset L(S_1) \vee L(S_2)$. Since $L(S_1) \vee L(S_2)$ is abelian and $B_0$ is maximal abelian, we conclude that $B_0 = L(S_1) \vee L(S_2)$. Thus, $B_i = L(S_i)$ for $i=1,2$. When $g \in S_2^\perp$, the unitary $u_g$ commutes with $L(S_2)$, but also with $L(S_1)$ because $L(S_1) = B_1 \subset L(H_1)$ and $L(H_1)$ is abelian. Since $B_0$ is maximal abelian, we get that $g \in S_1$. So, $S_1 = S_2^\perp$ and similarly $S_2 = S_1^\perp$.

The next step of the proof is to show that $\Lambda_0$ is abelian, i.e.\ that $\Lambda_1$ and $\Lambda_2$ are commuting subgroups of $\Lambda_0$. Put $T_i = H_i / S_i$. Since $S_1 = S_2^\perp$ and $S_2 = S_1^\perp$, we have the canonical dense embeddings $T_1 \hookrightarrow \Sth$ and $T_2 \hookrightarrow \Soh$. Viewing $L^\infty(\Soh \times \Sth) = L(S_1) \vee L(S_2)$ as a Cartan subalgebra of $A_0$, the associated equivalence relation is given by the orbits of the action
$$T_1 \times T_2 \actson \Soh \times \Sth : (g,h) \cdot (y,z) = (h \cdot y, g \cdot z) \; ,$$
where the actions on the right, namely $T_1 \actson \Sth$ and $T_2 \actson \Soh$, are given by translation. But viewing $B_0 = L(S_1) \vee L(S_2)$, the same equivalence relation is given by the orbits of the action $\Lambda_0 \actson \Soh \times \Sth$. We denote by $\om : (T_1 \times T_2) \times (\Soh \times \Sth) \recht \Lambda_0$ the associated $1$-cocycle. By construction, for all $g \in H_1$, $h \in H_2$ and $s \in \Lambda_0$, the support of $E_B(v_s^* u_g u_h)$ is the projection in $L^\infty(\Soh \times \Sth)$ given by the set
$$\{(y,z) \in \Soh \times \Sth \mid \om((gS_1,hS_2),(y,z)) = s \} \; .$$

Since $L(H_1) = B_1 \rtimes \Lambda_1$, for all $g \in T_1$, the map $(y,z) \mapsto \om((g,e),(y,z))$ only depends on the first variable and takes values in $\Lambda_1$ a.e. Reasoning similarly for $h \in T_2$, we find $\om_i : T_i \times \widehat{S_i} \recht \Lambda_i$ such that
$$\om((g,e),(y,z)) = \om_1(g,y) \quad\text{and}\quad \om((e,h),(y,z)) = \om_2(h,z) \quad\text{a.e.}$$
Writing $(g,h) = (g,e) (e,h)$ and $(g,h) = (e,h)(g,e)$, the $1$-cocycle relation implies that
\begin{equation}\label{eq.crucial-one}
\om_1(g,h \cdot y) \, \om_2(h,z) = \om_2(h,g \cdot z) \, \om_1(g,y)
\end{equation}
for all $g \in T_1$, $h \in T_2$ and a.e.\ $y \in \Soh$, $z \in \Sth$.

Define the subgroup $G_1 < \Lambda_1$ given by
$$G_1 = \{s \in \Lambda_1 \mid \forall t \in \Lambda_2 , tst^{-1} \in \Lambda_1\} \; .$$
Similarly, define $G_2 < \Lambda_2$. Note that $G_1$ and $G_2$ are normal subgroups of $\Lambda_0$. Since $\Lambda_0 = \Lambda_1 \Lambda_2$ and $\Lambda_1 \cap \Lambda_2 = \{e\}$, we also have that $G_1$ and $G_2$ commute. Rewriting \eqref{eq.crucial-one} as
$$\om_1(g,y) \, \om_2(h,z) = \om_2(h,g \cdot z) \, \om_1(g,h^{-1} \cdot y) \; ,$$
we find that for all $g \in T_1$, $h \in T_2$ and a.e.\ $y,y' \in \Soh$, $z \in \Sth$
$$\om_2(h,z)^{-1} \, \om_1(g,y')^{-1} \, \om_1(g,y) \, \om_2(h,z) \in \Lambda_1 \; .$$
Since $L(H_2) = B_2 \rtimes \Lambda_2$, the essential range of $\om_2$ equals $\Lambda_2$. It thus follows that
$$\om_1(g,y')^{-1} \, \om_1(g,y) \in G_1$$
for all $g \in T_1$ and a.e.\ $y,y' \in \Soh$. For every $g \in T_1$, we choose $\delta_1(g) \in \Lambda_1$ such that $\om_1(g,y) = \delta_1(g)$ on a non-negligible set of $y \in \Soh$. We conclude that $\om_1(g,y) = \delta_1(g) \, \mu_1(g,y)$ with $\mu_1(g,y) \in G_1$ a.e. We similarly decompose $\om_2(h,z) = \delta_2(h) \, \mu_2(h,z)$.

With these decompositions of $\om_1$ and $\om_2$ and using that $G_1, G_2$ are commuting normal subgroups of $\Lambda_0$, it follows from \eqref{eq.crucial-one} that for all $g \in T_1$, $h \in T_2$, the commutator $\delta_2(h)^{-1} \delta_1(g)^{-1} \delta_2(h) \delta_1(g)$ belongs to $G_1 G_2$, so that it can be uniquely written as $\eta_1(g,h) \eta_2(g,h)^{-1}$ with $\eta_i(g,h) \in G_i$. It then follows from \eqref{eq.crucial-one} that
\begin{align}
\mu_1(g,h\cdot y) &= \delta_2(h) \, \eta_1(g,h) \, \mu_1(g,y) \, \delta_2(h)^{-1} \;\; ,\label{eq.muone}\\
\mu_2(h,g \cdot z) &= \delta_1(g) \, \eta_2(g,h) \, \mu_2(h,z) \, \delta_1(g)^{-1} \;\; ,\notag
\end{align}
almost everywhere. Since $S_1 < H_1$ is torsion free, $\Soh$ has no finite quotients and thus no proper closed finite index subgroups. It follows that finite index subgroups of $T_2$ act ergodically on $\Soh$. We claim that for every $g \in T_1$, the map $y \mapsto \mu_1(g,y)$ is essentially constant. To prove this claim, fix $g \in T_1$ and denote $\xi : \Soh \recht G_1 : \xi(y) = \mu_1(g,y)$. For every $h \in T_2$, define the permutation
$$\rho_h : G_1 \recht G_1 : \rho_h(s) = \delta_2(h) \, \eta_1(g,h) \, s \, \delta_2(h)^{-1} \; .$$
So, \eqref{eq.muone} says that $\xi(h \cdot y) = \rho_h(\xi(y))$ for all $h \in T_2$ and a.e.\ $y \in \Soh$. Defining $V_1 \subset G_1$ as the essential range of $\xi$, it follows that $\{\rho_h\}_{h \in T_2}$ is an action of $T_2$ on $V_1$. The push forward via $\xi$ of the Haar measure on $\Soh$ is a $\{\rho_h\}_{h \in T_2}$-invariant probability measure on the countable set $V_1$ and has full support. It follows that all orbits of the action $\{\rho_h\}_{h \in T_2}$ on $V_1$ are finite. Choosing $s \in V_1$, the set $\xi^{-1}(\{s\}) \subset \Soh$ is non-negligible and globally invariant under a finite index subgroup of $T_2$. It follows that $\xi(y) = s$ for a.e.\ $y \in \Soh$, thus proving the claim.

Similarly, for every $h \in T_2$, the map $z \mapsto \mu_2(h,z)$ is essentially constant. So we have proved that $\om_1(g,y) = \delta_1(g)$ and $\om_2(h,z) = \delta_2(h)$ a.e. But then, \eqref{eq.crucial-one} implies that $\Lambda_1$ and $\Lambda_2$ commute, so that $\Lambda_0$ is an abelian group.

Since $A_0$ is a factor, $B_0^{\Lambda_0} = \C 1$ and thus $L(\Lambda_0) \subset A_0$ is maximal abelian. Since $L(\Lambda_0) = L(\Lambda_1) \vee L(\Lambda_2)$ with $L(\Lambda_i) \subset L(H_i)$, the same reasoning as with $B_i \subset L(H_i)$, using \eqref{eq.criterion-commute}, gives us subgroups $T_i \subset H_i$ such that $L(\Lambda_i) = L(T_i)$ and $T_1 = T_2^\perp$, $T_2 = T_1^\perp$.
Since $L(H_i) = B_i \rtimes \Lambda_i$ with $B_i = L(S_i)$ and $L(\Lambda_i) = L(T_i)$, we get that $H_i = S_i \oplus T_i$.

So far, we have proved that $B_0 = L(S_1) \vee L(S_2)$ and $L(\Lambda_0) = L(T_1) \vee L(T_2)$. In any crossed product $B_0 \rtimes \Lambda_0$ by a faithful action, the only unitaries in $L(\Lambda_0)$ that normalize $B_0$ are the multiples of the canonical unitaries $\{v_s\}_{s \in \Lambda}$. Therefore, $\T T_1 T_2 = \T \Lambda_0$. We have thus proved that the gms decomposition $(B_0,\Lambda_0)$ is identical to the gms decomposition $(L(S_1) \vee L(S_2), T_1 T_2)$.

Since $\Lambda_0, H_1$ and $H_2$ are globally $\{\beta_g\}_{g \in \Gamma_0}$-invariant, it follows that $T_i$ is a globally $\SL(3,\Z)$-invariant subgroup of $H_i$. Thus, $T_i = Q_i^3$ for some subgroup $Q_i < L_i$. Since $B_0, H_1$ and $H_2$ are globally $\Gamma_0$-invariant, it follows in the same way that $S_i = P_i^3$ for some subgroups $P_i < L_i$. Then, $L_i = P_i \oplus Q_i$ and $P_1,P_2$, as well as $Q_1,Q_2$, are each other's orthogonal complement under $\Om$. So, $L_1 = P_1 \oplus Q_1$ is an admissible direct sum decomposition.
\end{proof}

We now combine Proposition \ref{prop.compute-invariant-gms-decompositions} with Theorem \ref{main-SV} and Proposition \ref{prop-final}. We fix once and for all $\Gamma = \F_\infty$, $\Gamma_0 = \SL(3,\Z)$ and a surjective homomorphism $\beta : \Gamma \twoheadrightarrow \Gamma_0$ so that the automorphism $g \mapsto (g^{-1})^T$ of $\Gamma_0$ lifts to an automorphism of $\Gamma$. An obvious way to do this is by enumerating $\Gamma_0 = \{g_0,g_1,\ldots\}$ and defining $\beta : \Gamma \recht \Gamma_0$ by $\beta(s_i) = g_{i}$ for $i \geq 0$, where $(s_i)_{i \in \N}$ are free generators of $\Gamma$. Note that $\Ker \beta$ is automatically nontrivial.

We also fix countable abelian torsion free groups $L_1,L_2$ and a dense embedding $L_1 \hookrightarrow \widehat{L_2}$. Put $H_i = L_i^3$ and let $\Gamma$ act on $H_1 \hookrightarrow \widehat{H_2}$ through $\beta$. Then define $A_0 = L^\infty(\widehat{H_2}) \rtimes H_1$ together with the natural action $\beta : \Gamma \actson (A_0,\tau_0)$. Put $(A,\tau) = (A_0,\tau_0)^\Gamma$ with the action $\Gamma \times \Gamma \actson (A,\tau)$ given by $(g,h) \cdot \pi_k(a) = \pi_{gkh^{-1}}(\beta_h(a))$ for all $g,h,k \in \Gamma$, $a \in A_0$. Write $M = A \rtimes (\Gamma \times \Gamma)$.

We call an automorphism of $L_1$ admissible if it extends to a continuous automorphism of $\widehat{L_2}$. We call an isomorphism $\theta : L_1 \recht L_2$ admissible if it extends to a continuous isomorphism $\widehat{L_2} \recht \widehat{L_1}$.

\begin{thm} \label{thm.list-of-all-gms}
Whenever $L_1 = P_1 \oplus Q_1$ is an admissible direct sum decomposition with corresponding $L_2 = P_2 \oplus Q_2$, we define $B(P_1,Q_1) := (L(P_1^3) \vee L(P_2^3))^\Gamma$ and $\Lambda(P_1,Q_1) = (Q_1^3 \oplus Q_2^3)^{(\Gamma)} \rtimes (\Gamma \times \Gamma)$.

\begin{itemlist}
\item Every $B(P_1,Q_1), \Lambda(P_1,Q_1)$ gives a gms decomposition of $M$.

\item Every gms decomposition of $M$ is unitarily conjugate with a $B(P_1,Q_1), \Lambda(P_1,Q_1)$ for a unique admissible direct sum decomposition $L_1 = P_1 \oplus Q_1$.

\item Let $L_1 = P_1 \oplus Q_1$ and $L_1 = P_1' \oplus Q_1'$ be two admissible direct sum decompositions with associated gms decompositions $(B,\Lambda)$ and $(B',\Lambda')$.
\begin{itemlist}
\item $(B,\Lambda)$ and $(B',\Lambda')$ are conjugate by an automorphism of $M$ if and only if there exists an admissible automorphism $\theta : L_1 \recht L_1$ with $\theta(P_1) = P_1'$, $\theta(Q_1) = Q_1'$, or an admissible isomorphism $\theta : L_1 \recht L_2$ with $\theta(P_1) = P_2'$, $\theta(Q_1) = Q_2'$.

\item The Cartan subalgebras $B$ and $B'$ are unitarily conjugate if and only if $P_1 = P_1'$.

\item The Cartan subalgebras $B$ and $B'$ are conjugate by an automorphism of $M$ if and only if there exists an admissible automorphism $\theta : L_1 \recht L_1$ with $\theta(P_1) = P_1'$ or an admissible isomorphism $\theta : L_1 \recht L_2$ with $\theta(P_1) = P_2'$.
\end{itemlist}
\end{itemlist}
\end{thm}

\begin{proof}
Because of Proposition \ref{prop.compute-invariant-gms-decompositions}, Theorem \ref{main-SV} and Proposition \ref{prop-final}, it only remains to describe all automorphisms $\psi : A_0 \recht A_0$ that normalize the action $\beta : \Gamma \actson A_0$. This action $\beta$ is defined through the quotient homomorphism $\Gamma \twoheadrightarrow \Gamma_0$. Every automorphism of $\Gamma_0 = \SL(3,\Z)$ is, up to an inner automorphism, either the identity or $g \mapsto (g^{-1})^T$. So, we only need to describe all automorphisms $\psi : A_0 \recht A_0$ satisfying either $\psi \circ \beta_g = \beta_g \circ \psi$ for all $g \in \Gamma_0$, or $\psi \circ \beta_g = \beta_{(g^{-1})^T} \circ \psi$.

In the first case, reasoning as in the first paragraphs of the proof of Proposition \ref{prop.compute-invariant-gms-decompositions}, we get that $\psi(L(H_i)) = L(H_i)$ for $i=1,2$. So, for every $g \in H_1$, $\psi(u_g)$ is a unitary in $L(H_1)$ that normalizes $L(H_2)$. This forces $\psi(u_g) \in \T H_1$ and we conclude that $\psi(\T H_1) = \T H_1$. Similarly, $\psi(\T H_2) = \T H_2$. In the second case, we obtain in the same way that $\psi(\T H_1) = \T H_2$ and $\psi(\T H_2) = \T H_1$. The further analysis is analogous in both cases and we only give the details of the first case.

We find automorphisms $\theta_i : H_i \recht H_i$ such that $\psi(u_g) \in \T u_{\theta_i(g)}$ for all $i = 1,2$ and $g \in H_i$. Since $\theta_i$ commutes with the action of $\SL(3,\Z)$ on $H_i$, we get that $\theta_1 = \theta^3$, where $\theta : L_1 \recht L_1$ is an admissible automorphism. It follows that $\psi$ maps the gms decomposition associated with $L_1 = P_1 \oplus Q_1$ to the gms decomposition associated with $L_1 = \theta(P_1) \oplus \theta(Q_1)$. This concludes the proof of the theorem.
\end{proof}

The following concrete examples provide a proof for Theorem \ref{thmA}.

\begin{thm}\label{thm.exactly-n}
For all $n \geq 1$, consider the following two embeddings $\pi_i : \Z^n \hookrightarrow \T^{2n}$.
\begin{itemlist}
\item $\pi_1(k)= (\al_1^{k_1},\al_2^{k_1},\ldots,\al_{2n-1}^{k_n},\al_{2n}^{k_n})$ for rationally independent irrational angles $\al_j \in \T$.
\item $\pi_2(k)= (\al^{k_1},\be^{k_1},\ldots,\al^{k_n},\be^{k_n})$ for rationally independent irrational angles $\al,\be \in \T$.
\end{itemlist}
Applying Theorem \ref{thm.list-of-all-gms} to the embeddings $\pi_1$ and $\pi_2$, we respectively obtain
\begin{itemlist}
\item a II$_1$ factor $M$ that has exactly $2^n$ gms decompositions up to unitary conjugacy, and with the associated $2^n$ Cartan subalgebras non conjugate by an automorphism of $M$~;
\item a II$_1$ factor $M$ that has exactly $n+1$ gms decompositions up to conjugacy by an automorphism of $M$, and with the associated $n+1$ Cartan subalgebras non conjugate by an automorphism of $M$.
\end{itemlist}
\end{thm}
\begin{proof}
Whenever $\cF \subset \{1,\ldots,n\}$, we have the direct sum decomposition $\Z^n = P(\cF) \oplus P(\cF^c)$ where $P(\cF) = \{x \in \Z^n \mid \forall i \not\in \cF, x_i = 0 \}$.

In the case of $\pi_1$, these are exactly all the admissible direct sum decompositions of $\Z^n$. Also, the only admissible automorphisms of $\Z^n$ are the ones of the form $(x_1,\ldots,x_n) \mapsto (\eps_1 x_1,\ldots,\eps_n x_n)$ with $\eps_i = \pm 1$. Since $\Z^n \not\cong \Z^{2n}$, there are no isomorphisms ``exchanging $L_1$ and $L_2$''.

In the case of $\pi_2$, all direct sum decompositions and all automorphisms of $\Z^n$ are admissible. For every direct sum decomposition $\Z^n = P_1 \oplus Q_1$, there exists a unique $k \in \{0,\ldots,n\}$ and an automorphism $\theta \in \GL(n,\Z)$ such that $\theta(P_1) = P(\{1,\ldots,k\})$ and $\theta(Q_1) = P(\{k+1,\ldots,n\})$. Again, there are no isomorphisms exchanging $L_1$ and $L_2$. So, the $n+1$ direct sum decompositions $\Z^n = P(\{1,\ldots,k\}) \oplus P(\{k+1,\ldots,n\})$, $0 \leq k \leq n$, exactly give the possible gms decompositions of $M$ up to conjugacy by an automorphism of $M$. When $k \neq k'$, there is no isomorphism $\theta \in \GL(n,\Z)$ with $\theta(P(\{1,\ldots,k\})) = P(\{1,\ldots,k'\})$. Therefore, the $n+1$ associated Cartan subalgebras are nonconjugate by an automorphism either.
\end{proof}

\begin{rem}
When $L_1,L_2$ are torsion free abelian groups and $L_1 \hookrightarrow \widehat{L_2}$ is a dense embedding, then the set of admissible homomorphisms $L_1 \recht L_1$ is a ring $\cR$ that is torsion free as an additive group. The admissible direct sum decompositions of $L_1$ are in bijective correspondence with the idempotents of $\cR$. As a torsion free ring, $\cR$ either has infinitely many idempotents, or finitely many that are all central, in which case their number is a power of $2$. So, the number of gms decompositions (up to unitary conjugacy) of the II$_1$ factors produced by  Theorem \ref{thm.list-of-all-gms} is always either infinite or a power of $2$.
\end{rem}

\begin{rem}
Still in the context of Theorem \ref{thm.list-of-all-gms}, we call a subgroup $P_1 < L_1$ admissible if $L_1 \cap \overline{P_1} = P_1$, where $\overline{P_1}$ denotes the closure of $P_1$ inside $\widehat{L_2}$. Note that $P_1 < L_1$ is admissible if and only if there exists a subgroup $P_2 < L_2$ such that $P_2 = P_1^\perp$ and $P_1 = P_2^\perp$. Whenever $P_1 < L_1$ is an admissible subgroup, we define $B(P_1) := (L(P_1^3) \vee L(P_2^3))^\Gamma$. It is easy to check that all $B(P_1)$ are Cartan subalgebras of $M$ and that $B(P_1)$ is unitarily conjugate with $B(P_1')$ if and only if $P_1 = P_1'$.

It is highly plausible that these $B(P_1)$ describe all Cartan subalgebras of $M$ up to unitary conjugacy. We could however not prove this because all our techniques make use of the dual coaction associated with a gms decomposition of $M$.

Also note that an admissible subgroup $P_1 < L_1$ cannot necessarily be complemented into an admissible direct sum decomposition $L_1 = P_1 \oplus Q_1$. In such a case, $B(P_1)$ is a Cartan subalgebra of $M$ that is not of group measure space type. However, $M$ can be written as a \emph{cocycle} crossed product of $B(P_1)$ by an action of $\Lambda(P_1) = \bigl( (L_1/P_1)^3 \oplus (L_2 / P_2)^3 \bigr)^{(\Gamma)} \rtimes (\Gamma \times \Gamma)$, but the cocycle is nontrivial.
\end{rem}

\end{document}